\documentclass[reqno,11pt]{amsart}
\usepackage{amsmath,amssymb,amsthm,amsfonts,mathrsfs,times,color}
\usepackage{ulem,latexsym,times,color,epsfig,bm}
\usepackage[colorlinks=true, pdfstartview=FitV, linkcolor=blue,citecolor=blue, urlcolor=blue]{hyperref}
\usepackage[capitalize,nameinlink,noabbrev]{cleveref}
\usepackage{amsmath,amsfonts,mathtools}
\usepackage{amsthm}
\usepackage{graphicx}
\usepackage{color}
\usepackage{comment}
\usepackage{multicol}
\DeclareMathOperator{\sign}{sign}

\newtheorem{thm}{Theorem}[section]

\newtheorem{notation}{Notation}[section]
\newtheorem{lemma}[thm]{Lemma}
\newtheorem{remark}{Remark}[section]
\newtheorem{theorem}[thm]{Theorem}

\numberwithin{equation}{section}
\newcommand{\veps}{{\varepsilon}}
\newcommand{\kap}{{\kappa}}
\newcommand{\tu}{{\tilde u}}

\newcommand{\tv}{{\tilde v}}

\newcommand{\gs}{\geqslant}
\newcommand{\ls}{\leqslant}

\newcommand{\dx}{{\mathrm{d}x}}

\usepackage{epsfig}
\usepackage{graphicx}
\usepackage{colordvi}
\usepackage{graphics}
\usepackage{color}
\usepackage{multirow}

\makeatletter
\@namedef{subjclassname@2020}{%
  \textup{2020} Mathematics Subject Classification}
\makeatother

\begin{document}

\title[On a chemotaxis-growth system with unequal dynamic Dirichlet data]{Boundary symmetry breaking via logistic damping in a chemotaxis-growth system}

\author[Y. Chen]{Yiren Chen}
\address[Y. Chen]{School of Mathematical Sciences, Shenzhen University, Shenzhen 518060, China}
\email{yrchen@szu.edu.cn}

\author[P. Fuster Aguilera]{Padi Fuster Aguilera}
\address[P. Fuster Aguilera]{Department of Mathematics\\
University of Colorado Boulder\\ Campus Box 395, Boulder, CO, 80309, USA}
\email{padi.fuster@colorado.edu}

\author[V. Martinez]{Vincent Martinez}
\address[V. Martinez]{Department of Mathematics and Statistics, CUNY Hunter College\\
New York, NY 10065, USA\\
Department of Mathematics \\
CUNY Graduate Center \\
New York, NY 10016
}
\email{vrmartinez@hunter.cuny.edu}

\author[K. Zhao]{Kun Zhao}
\address[K. Zhao]{School of Mathematical Sciences, Harbin Engineering University, Harbin 150001, China}
\email{kzhao@hrbeu.edu.cn}


\begin{abstract}
We establish global stability for a chemotaxis-growth model with logarithmic sensitivity under dynamic Dirichlet boundary conditions on a 1D domain. We analyze both parabolic-parabolic and parabolic-hyperbolic systems. The key challenge is handling time-dependent boundary data for the unknown functions. We overcome this by introducing dynamic reference profiles which suitably interpolate boundary values. Using an expanded entropy functional measuring deviation from these profiles, we prove 
energy estimates the uniform boundedness of solutions and global asymptotic stability of perturbations. 
\end{abstract}
\maketitle

\section{Introduction}
Chemotaxis, which is the directed movement of organisms along chemical gradients, is fundamental to many essential biological processes, including bacterial pattern formation, embryonic development, immune cell trafficking, and many more. A fundamental mathematical model of this phenomenon is given by the Keller-Segel system \cite{KS}:
\begin{align}\label{eq:keller:segel}
    \begin{split}
    u_t&=-\nabla\cdot(-D\nabla u+\chi(u,c)
    \nabla c)+f(u,c),\\
    c_t&=-\nabla\cdot(-\veps\nabla c)+g(u,c),
    \end{split}
\end{align}
where $u(x,t)\geqslant 0$ denotes the population density of the organism, and $c(x,t)\geqslant 0$ the chemical concentration. The chemotactic flux, represented by $\chi(u,c)\nabla c$, 
 asserts that while chemotactic motion should follow the negative gradient of the chemical potential, it should, in general, be sensitive to the local density of both the organism and chemical; this sensitivity is modeled by $\chi$. 
The positive parameter $D$ represents the diffusivity or the organism density,  while $\veps$ represents the diffusivity of the chemical.

This paper focuses on a specific variant of \eqref{eq:keller:segel} in which cell density dynamics is coupled with a chemical signal in 1D that simultaneously features chemotaxis with logarithmic sensitivity \cite{stevens1997aggregation} \textit{and} logistic cell growth.  In particular, we will assume that
    \begin{align}\label{eq:assumptions}
      \chi(u,c)=\chi \frac{u}c,\quad f(u,c)=\kappa_1u\left(1-\frac{u}{\kappa_2}\right),\quad g(u,c)=\mu u^\gamma c-\sigma c,
    \end{align}
 {where $\kap_1,\kap_2>0$, $\mu,\sigma,\chi, \gamma \in\mathbb{R}$; $\chi$ is attractant when ($\chi>0$) and repulsive when $\chi<0$), while $|\chi|$  {measures} the strength of chemotactic sensitivity. Notice that under these assumptions, $\chi(u,c)\nabla c=\chi u\nabla \phi(c)$, where $\phi(c)=\ln c$}. 

By applying the transformations $\hat c=e^{\sigma t}c$ (which simplifies the linear term $\sigma c$), and $\hat u= u/\kappa_2$, and the standard Cole-Hopf transformation for the resulting system, i.e., setting $v=(\ln \hat c)_x$. Avoiding the hat notation, we arrive at
\begin{equation}\label{eq:main1}
\begin{aligned}
    u_t+\chi(u v)_x&=D u_{xx}+ \kappa_1u(1-u)\\
    v_t+\left(\mu\kappa_2^\gamma u^\gamma -\veps v^2\right)_x&=\veps v_{xx}.
\end{aligned}
\end{equation}
 For our purposes, we will assume that $\chi\mu>0$; this guarantees that the system is hyperbolic, given that in the case where $\chi\mu<0$, Levine-Sleeman \cite{LS} demonstrate that solutions can blow-up in finite time. 

 We note that by rescaling
 \[
 \hat v= \sqrt{\frac{\chi}{\mu \kappa_2^\gamma}} v, \quad \hat x= \sign(\chi)\frac{\sqrt{\chi \mu\kappa_2^\gamma}}{D}x, \quad \hat t=  \frac{\mu \kappa_2^\gamma \chi}{D} t,
 \]
 we get to the non-dimensional system (upon dropping the hat notation),
\begin{equation}\label{eq:main2}
\begin{aligned}
    u_t+\sign(\chi)(u v)_x&= u_{xx}+ r u(1-u)\\
    v_t+\sign(\chi)(u^\gamma -\frac{\varepsilon}{\chi} v^2)_x&= \frac{\varepsilon}{D}  v_{xx},
\end{aligned}
\end{equation}
where $r=D\kappa_1/\mu \kappa_2^\gamma \chi$. Since we will not be concerned with studying singular limits, we will assume $\chi=-1, D=r=1$, and either $\veps=1$ or $\veps=0$, for the convenience of the analysis.
Thus, the governing system that we are analysing in this paper comprises reaction-diffusion-advection equations:
\begin{equation}\label{0.1}
\begin{aligned}
u_t - (vu)_x &= u_{xx} + u(1-u), \quad &x&\in \mathbb{R},\quad t>0,\\
v_t - (v^2)_x &= v_{xx} + (u^\gamma)_x, &x&\in \mathbb{R},\quad t>0,
\end{aligned}
\end{equation}
or its hyperbolic counterpart, that is, when $\veps=0$ in \eqref{eq:main2}:
\begin{equation}\label{0.2}
\begin{aligned}
u_t - (u v)_x &= u_{xx} + u(1-u), \quad &x&\in \mathbb{R},\quad t>0,\\
v_t - (u^\gamma)_x &= 0, &x&\in \mathbb{R},\quad t>0.
\end{aligned}
\end{equation}

Analyzing systems \eqref{0.1} and \eqref{0.2}, particularly concerning global existence, boundedness, and asymptotic stability of solutions for large initial data, is challenging due to the potentially destabilizing advection-like terms. While significant progress exists for bounded domains, most results depend critically on boundary conditions. Static boundary conditions, modeling fixed quantities or fluxes (e.g., impermeable walls or constant reservoirs), are the most studied scenario \cite{FFH,DL1,LS,LSN,1d2,1d3,MD1,1d6,1d7,ZZ}.

However, many physical scenarios critically involve dynamic interactions across boundaries such as the controlled influx/efflux of cells or chemicals in laboratory experiments \cite{KalininNeumannSourjikWu2010, ParkAminzare2021}, in clinical research (e.g. administering drug delivery in cancer therapeutics \cite{tele}), or interaction with reservoirs at the boundary whose concentrations might naturally change over time, for instance, due to seasonal variation. Such situations naturally call for time-dependent boundary conditions in order to properly model the subsequent chemotaxis. In this paper, we study the following Dirichlet type dynamic boundary conditions: 
\begin{align}\label{0.3}
u(a,t)=\alpha_1(t),\quad u(b,t)=\alpha_2(t), \quad v(a,t)=\beta_1(t),\quad v(b,t)=\beta_2(t), \quad t \geqslant 0,
\end{align}
where the boundary data $\alpha_i$ and $\beta_i$ ($i=1,2$) are prescribed time-dependent functions. In the presence of such boundary conditions, the system is constantly driven away from any fixed-boundary equilibrium, which subsequently complicates the adaptation of energy-type methods previously developed in the presence of constant boundary conditions.

Analysis of system \eqref{0.1} under dynamic Dirichlet conditions began for the no-growth case:
\begin{equation}\label{0.4}
\begin{aligned}
u_t - (u v)_x &= u_{xx}, \\
v_t - (u^\gamma)_x &= v_{xx} + (v^2)_x.
\end{aligned}
\end{equation}
Global stability of large-data classical solutions to \eqref{0.4} was established in \cite{1d5} for $\gamma=1$ and \cite{FXXZ} for $\gamma>1$ under equal boundary conditions: $u(a,t)=u(b,t)$, $v(a,t)=v(b,t)$, with suitable integrability assumptions on the boundary data and derivatives. This was later generalized to scenarios where $v(a,t) \neq v(b,t)$ \cite{FZ} and to the hyperbolic counterpart \cite{FXXZ,FZ}:
\begin{equation*}
\begin{aligned}
u_t - (u v)_x &= u_{xx}, \\
v_t - (u^\gamma)_x &= 0.
\end{aligned}
\end{equation*}
However, the biologically crucial case of unequal cell density boundary values ($u(a,t) \neq u(b,t)$), essential for modeling asymmetric cell flows or chemical gradients across a domain (e.g., chemotaxis assays, tissue interfaces, directed migration), remains largely unexplored.

This paper bridges this gap, establishing the global stability of solutions for both the parabolic-parabolic \eqref{0.1} and parabolic-hyperbolic \eqref{0.2} chemotaxis-growth systems on a bounded interval $[a, b]$, under general time-dependent Dirichlet boundary conditions \eqref{0.3}, including the important physical case $u(a,t) \neq u(b,t)$. The core challenge is handling the dynamic boundary forcing. To overcome this, we introduce dynamically evolving reference profiles $\alpha(x,t)$ and $\beta(x,t)$ (or $\Psi(t)$ for the hyperbolic case) that linearly interpolate the boundary data for $u$ and $v$ at each time $t$. Our analysis centers on a carefully designed, expanded entropy functional that measures deviations from these dynamic references, which ultimately allows us to carry out an energy-method type analysis in a novel way to rigorously establish the following results:
\begin{itemize}
\item[1.] Uniform boundedness: solution remains bounded in Sobolev norms for all $t>0$. Biologically, this ensures cell density never blows up to infinity, maintaining physically realistic levels within the system despite dynamic boundary inputs.

\item[2.] Time integrability: higher-order spatial and temporal derivatives are integrable over $[0, \infty)$. Mathematically, this provides crucial control over solution variations over time. Biologically, it implies that rapid, uncontrolled fluctuations dampen out, leading to smoother dynamics.

\item[3.] Asymptotic stability: the deviations $\tilde u = u - \alpha$ and $\tilde v = v - \beta$ (or $\tilde v = v - \Psi$) converge to zero as $t\to\infty$. This is the central biological result: It demonstrates that the system's internal state asymptotically aligns with the dynamically imposed boundary conditions. Regardless of initial internal distributions and complex chemotactic interactions, the solution eventually evolves to match the profile linearly interpolated between the time-varying boundary values. The system is globally attracted to the state defined by its dynamic boundaries.
\end{itemize}
These results hold under natural assumptions: positivity of cell density boundary data ($\alpha_1(t)$, $\alpha_2(t) > 0$) and suitable integrability conditions on the boundary data and their derivatives \eqref{BA}.

To our knowledge, this work provides the first comprehensive global stability analysis for systems \eqref{0.1} and \eqref{0.2} under general dynamic Dirichlet conditions, fully relaxing the constraint of equal cell density at boundaries. The developed techniques -- notably the use of dynamic reference profiles, tailored expanded entropy, and a domain splitting argument -- offer a potentially powerful framework for analyzing diverse reaction-diffusion systems subject to time-dependent boundary forcing, relevant to many biological and physical contexts.

The rest of this paper is organized as follows: Section \ref{PP} details the analysis and energy estimates for the parabolic-parabolic system \eqref{0.1}, establishing global stability. Section \ref{PH} adapts the methodology to the parabolic-hyperbolic system \eqref{0.2}, addressing the reduced dissipation and deriving analogous stability results. The paper ends with concluding remarks in Section \ref{C}.

\begin{notation}
Throughout this paper, we denote the norms of $L^p((a,b))$, $p\in [1,\infty]$, and $H^k((a,b))$ by $\|\cdot\|_{L^p}$ and $\|\cdot\|_{H^k}$, respectively. Unless otherwise specified, $C > 0$ represents a generic constant, independent of the solution, whose value may vary from line to line.
\end{notation}

\section{Parabolic-Parabolic System under Dynamic Dirichlet B.C.}\label{PP}

In this section, we analyze the dynamics of system \eqref{0.1} under the boundary conditions \eqref{0.3}. For the reader's convenience, we first present the initial-boundary value problem:
\begin{subequations}\label{1.1}
\begin{alignat}{5}
u_t - (u v)_x &= u_{xx} + u(1-u), \quad &x&\in (a,b),\ t>0, \label{1.1a}\\
v_t - (u^\gamma)_x &=  v_{xx} + (v^2)_x, \quad &x&\in (a,b),\ t>0; \label{1.1b}\\
(u,v)(x,0)&=(u_0,v_0)(x), \quad &x&\in [a,b], \label{1.2}\\
u(a,t)&=\alpha_1(t),\quad u(b,t)=\alpha_2(t), \quad &t& \geqslant 0, \label{1.3a}\\
v(a,t)&=\beta_1(t),\quad v(b,t)=\beta_2(t), \quad &t& \geqslant 0, \label{1.3b}
\end{alignat}
\end{subequations}
where $u_0$, $v_0$ are given initial data whose properties will be specified later. To state the main results, we define the reference profiles interpolating the boundary data:
\begin{subequations}\label{1.4}
\begin{alignat}{2}
\alpha(x,t) &\triangleq \alpha_1(t) + \frac{x-a}{b-a}\left[\alpha_2(t)-\alpha_1(t)\right], \quad &x&\in[a,b],\ t\geqslant 0,\label{1.4a}\\
\beta(x,t) &\triangleq \beta_1(t) + \frac{x-a}{b-a}\left[\beta_2(t)-\beta_1(t)\right],\quad &x&\in[a,b],\ t\geqslant 0.\label{1.4b}
\end{alignat}
\end{subequations}
Our main theorem is stated as follows.

\begin{theorem}\label{thm1}
Consider the initial-boundary value problem \eqref{1.1} with $\gamma\geqslant 1$. Suppose that the initial data satisfy $u_0 > 0$, $u_0,v_0 \in H^1((a,b))$, and are compatible with the boundary conditions. Assume that the boundary data $\alpha_1$, $\alpha_2$, $\beta_1$, $\beta_2$ are smooth functions on $\mathbb{R}_+=[0,\infty)$, satisfying
\begin{align}
\alpha_1(t) \geqslant  \underline{\alpha_1}, \quad \alpha_2(t) \geqslant  \underline{\alpha_2}, \quad \text{and} \quad \alpha_1-1,\alpha_2-1,\beta_1,\beta_2 \in W^{1,1}(\mathbb{R}_+),\label{BA}
\end{align}
where $\underline{\alpha_1}$, $\underline{\alpha_2}>0$ are constants. Then there exists a unique solution to the IBVP, such that for $\forall\,t>0$,
\begin{align*}
\|(u-\alpha)(t)\|^2_{H^1}+\|(v-\beta)(t)\|^2_{H^1} + \int_0^t \left(\|(u-\alpha)(\tau)\|^2_{H^2}+\|(v-\beta)(\tau)\|^2_{H^2}\right)\mathrm{d}\tau \leqslant C,
\end{align*}
where $\alpha$ and $\beta$ are defined by \eqref{1.4}. Moreover, the solution satisfies
\begin{align}\label{AB1}
\|(u-\alpha)(t)\|_{H^1}+\|(v-\beta)(t)\|_{H^1} \to 0\quad\text{as}\quad t\to\infty.
\end{align}
\end{theorem}

\begin{remark}
Theorem \ref{thm1} does not impose smallness restrictions on the initial or boundary data in the corresponding norm. Nevertheless, some degree of ``smallness" is enforced by our assumptions through integrability in $\mathbb{R}_+$. On the other hand, while the conditions in \eqref{BA} imply that the moduli of the boundary data are uniformly bounded in time, they may be large for a long transient period; this is one of the main sources of technical difficulties for us. The result \eqref{AB1} thus establishes the global asymptotic stability of the deviations of $u$ and $v$ from their (potentially large) reference profiles provided that the profiles at the boundary and its oscillations decay at infinity in an integrable fashion. Note that we will frequently make use of the uniform boundedness of the boundary profiles in time in the proof of the above theorem. 
\end{remark}

The proof of Theorem \ref{thm1} comprises three essential components: local well-posedness, {\it a priori} estimates, and a continuation argument. The first step can be established using standard techniques, such as Galerkin approximation and the contraction mapping principle, augmented by the energy estimates developed in this paper. The core of this section is devoted to establishing the {\it a priori} estimates for the local solution. Once these estimates are secured, the global well-posedness and long-time behavior of the solution follow via a standard continuation argument. In deriving the {\it a priori} estimates, two elements are essential: an expanded entropy functional and a domain splitting argument that is novel in this context. The former measures the solution's deviation from its dynamic reference profile, while the latter controls a quadratic nonlinearity arising from the dynamic boundary interpolation profile. Since the expanded entropy functional takes different forms when the value of $\gamma$ varies, we split the proof of the {\it a priori} estimates into two subsections for $\gamma=1$ and $\gamma > 1$.

\subsection{Proof of Theorem \ref{thm1} when $\gamma=1$}

We organize the proof into four subsections.

\subsubsection{Entropy-based estimates}

We define the expanded entropy functional:
\begin{align*}
\mathcal{E}_1(t) \triangleq \int_a^b \left[(u\ln u - u) - (\alpha \ln \alpha - \alpha) - \ln\alpha(u-\alpha)\right](x,t)\mathrm{d}x,
\end{align*}
where $\alpha$ is defined by \eqref{1.4a}. A direct calculation shows that 
\begin{align}\label{A2}
\frac{\mathrm{d}}{\mathrm{d}t} \mathcal{E}_1(t) + 4\|(\sqrt{u})_x\|_{L^2}^2 + \int_a^b u_xv\,\mathrm{d}x + \int_a^b u(u-\alpha)(\ln u - \ln\alpha)\mathrm{d}x = \sum_{k=1}^4 \mathsf{R}_{1,k},
\end{align}
where the quantities on the right-hand side are defined by
\begin{align*}
&\mathsf{R}_{1,1} \triangleq \int_a^b \frac{u_x\alpha_x}{\alpha}\mathrm{d}x, \qquad \mathsf{R}_{1,2} \triangleq \int_a^b \frac{uv\alpha_x}{\alpha}\mathrm{d}x, \qquad \mathsf{R}_{1,3} \triangleq - \int_a^b \frac{(u-\alpha)\alpha_t}{\alpha}\mathrm{d}x, \\
&\mathsf{R}_{1,4} \triangleq \int_a^b u(1-\alpha)(\ln u - \ln\alpha)\mathrm{d}x. 
\end{align*}
Using \eqref{1.4a}, \eqref{BA} and the positivity of $u$ (which follows from the positivity of the initial and boundary conditions, together with the maximum principle), we can show that
\begin{align}\label{A3}
|\mathsf{R}_{1,1}| \leqslant  2\|(\sqrt{u})_x\|_{L^2}^2 + C|\alpha_2-\alpha_1| \|u\|_{L^1}.
\end{align}
To control $\mathsf{R}_{1,2}$, first note that for any fixed $t>0$,
\begin{align*}
\int_a^b u(u-\alpha)(\ln u - \ln\alpha)\mathrm{d}x = \int_{\{u>\alpha\}} \frac{u}{\xi}(u-\alpha)^2\mathrm{d}x + \int_{\{u\leqslant \alpha\}} u(u-\alpha)(\ln u - \ln\alpha)\mathrm{d}x,
\end{align*}
where $\xi$ is between $u$ and $\alpha$. Since $u$ and $\alpha$ are positive, $u(u-\alpha)(\ln u - \ln\alpha)\geqslant 0$. Hence, 
\begin{align}\label{A5}
\int_a^b u(u-\alpha)(\ln u - \ln\alpha)\mathrm{d}x \geqslant \int_{\{x:\, u>\alpha\}} (u-\alpha)^2\mathrm{d}x.
\end{align}
Again, by virtue of the positivity of $u$ and $\alpha$,
\begin{align}\label{A6}
|\mathsf{R}_{1,2}| \leqslant \int_{\{x:\, u>\alpha\}}\frac{|u-\alpha||v||\alpha_x|}{\alpha}\mathrm{d}x + \int_{\{x:\, u>\alpha\}} |v||\alpha_x|\mathrm{d}x +\int_{\{x:\, u\leqslant\alpha\}}\frac{u|v||\alpha_x|}{\alpha}\mathrm{d}x.
\end{align}
By Schwarz inequality and boundedness of $\alpha$, we can show that
\begin{align}\label{A7}
\int_{\{x:\, u>\alpha\}}\frac{|u-\alpha||v||\alpha_x|}{\alpha}\mathrm{d}x \leqslant \frac12\int_{\{x:\, u>\alpha\}}(u-\alpha)^2\mathrm{d}x + C|\alpha_2-\alpha_1|\|v\|_{L^2}^2.
\end{align}
For the rest on the right-hand side of \eqref{A6}, we have 
\begin{align}\label{A8}
\int_{\{x:\, u>\alpha\}} |v||\alpha_x|\mathrm{d}x +\int_{\{x:\, u\leqslant \alpha\}}\frac{u|v||\alpha_x|}{\alpha}\mathrm{d}x \leqslant C|\alpha_2-\alpha_1| \left(\|v\|_{L^2}^2 + 1\right).
\end{align}
Using \eqref{A5}, \eqref{A7} and \eqref{A8}, we update \eqref{A6} as 
\begin{align}\label{A9}
|\mathsf{R}_{1,2}| \leqslant \frac12\int_a^b u(u-\alpha)(\ln u - \ln\alpha)\mathrm{d}x + C|\alpha_2-\alpha_1| \left(\|v\|_{L^2}^2 + 1\right).
\end{align}
For $\mathsf{R}_{1,3}$, using the definition of $\alpha$, we infer that 
\begin{align}\label{A10}
|\mathsf{R}_{1,3}| \leqslant C\left(|\alpha_1'|+|\alpha_2'|\right) \left(\|u\|_{L^1} + 1\right).
\end{align}
To bound $\mathsf{R}_{1,4}$, note that since $u$ and $\alpha$ are positive, 
\begin{align*}
0\leqslant \left[(u\ln u - u) - (\alpha \ln \alpha - \alpha) - \ln\alpha(u-\alpha)\right] = u(\ln u - \ln\alpha) -  (u-\alpha),
\end{align*}
which implies
\begin{align}\label{A12}
 \int_a^b \left|u(\ln u - \ln\alpha)\right|\mathrm{d}x \leqslant \mathcal{E}_1 + \|u\|_{L^1} + \|\alpha\|_{L^1}.
\end{align}
Using \eqref{A12} and boundedness of $\alpha$, we can show that
\begin{align}\label{A13}
|\mathsf{R}_{1,4}| \leqslant C\left(|1-\alpha_1|+|\alpha_2-\alpha_1|\right)\left(\mathcal{E}_1 + \|u\|_{L^1} + 1\right).
\end{align}
Substituting \eqref{A3}, \eqref{A9}, \eqref{A10} and \eqref{A13} into \eqref{A2} gives
\begin{align}\label{A14}
&\ 2\frac{\mathrm{d}}{\mathrm{d}t} \mathcal{E}_1(t) + 4\|(\sqrt{u})_x\|_{L^2}^2 + 2\int_a^b u_xv\,\mathrm{d}x + \int_a^b u(u-\alpha)(\ln u - \ln\alpha)\mathrm{d}x \notag\\
\leqslant &\ C\left(|1-\alpha_1|+|\alpha_2-\alpha_1| + |\alpha_1'|+|\alpha_2'| \right)\left(\mathcal{E}_1 + \|u\|_{L^1} + \|v\|_{L^2}^2 +1\right).
\end{align}
Furthermore, according to \cite{1d5}, we have 
\begin{align*}
0< u \leqslant (u\ln u - u) - (\alpha \ln \alpha - \alpha) - \ln\alpha(u-\alpha) + (e-1)\alpha,
\end{align*}
which, together with the boundedness of $\alpha$, yields
\begin{align}\label{A16}
0< \|u\|_{L^1} \leqslant \mathcal{E}_1 + C.
\end{align}
Substituting \eqref{A16} into \eqref{A14}, we obtain
\begin{align}\label{A17}
&\ 2\frac{\mathrm{d}}{\mathrm{d}t} \mathcal{E}_1(t) + 4\|(\sqrt{u})_x\|_{L^2}^2 + 2\int_a^b u_xv\,\mathrm{d}x + \int_a^b u(u-\alpha)(\ln u - \ln\alpha)\mathrm{d}x \notag\\
\leqslant &\ C\left(|1-\alpha_1|+|\alpha_2-\alpha_1| + |\alpha_1'|+|\alpha_2'| \right)\left(\mathcal{E}_1 + \|v\|_{L^2}^2 +1\right).
\end{align}
In the next step, we utilize Eq.\,\eqref{1.1b} to compensate the third term on the left-hand side of \eqref{A17}.

Let $\tilde{v}(x,t) \triangleq v(x,t) - \beta(x,t)$, where $\beta$ is defined by \eqref{1.4b}. Then we deduce from \eqref{1.1b}:
\begin{align}\label{A18}
\tilde{v}_t - u_x = \tilde{v}_{xx} + (\tilde{v}^2)_x + 2(\tilde{v}\beta)_x + (\beta^2)_x - \beta_t.
\end{align}
Note that $\tilde{v}|_{x=a}=\tilde{v}|_{x=b}=0$. Taking the $L^2$ inner product of \eqref{A18} with $\tilde{v}$ yields 
\begin{align}\label{A21}
&\ \frac{\mathrm{d}}{\mathrm{d}t}\|\tilde{v}\|_{L^2}^2 + 2\|\tilde{v}_x\|_{L^2}^2 - 2\int_a^b u_xv\,\mathrm{d}x \notag\\
 = &\ 2\int_a^b \beta_x \tilde{v}^2\mathrm{d}x + 4\int_a^b \beta\beta_x\tilde{v}\,\mathrm{d}x - 2\int_a^b \beta_t\tilde{v}\,\mathrm{d}x + 2\int_a^b u \beta_x\mathrm{d}x + 2(\alpha_1\beta_1 - \alpha_2\beta_2).
\end{align}
Using Schwarz inequality, the boundary conditions and \eqref{A16}, we infer that 
\begin{align}\label{A22}
&\ \left| 2\int_a^b \beta_x \tilde{v}^2\mathrm{d}x + 4\int_a^b \beta\beta_x\tilde{v}\,\mathrm{d}x - 2\int_a^b \beta_t\tilde{v}\,\mathrm{d}x + 2\int_a^b u \beta_x\mathrm{d}x + 2(\alpha_1\beta_1 - \alpha_2\beta_2)\right| \notag\\
\leqslant &\ C\left( |\alpha_1-\alpha_2| + |\beta_2-\beta_1| + |\beta_1'| + |\beta_2'| \right) \left( \mathcal{E}_1 + \|\tilde{v}\|_{L^2}^2 + 1\right).
\end{align}
Updating \eqref{A21} by \eqref{A22}, then adding the result to \eqref{A17}, we get 
\begin{align*}
\frac{\mathrm{d}}{\mathrm{d}t} \left(2\mathcal{E}_1 + \|\tilde{v}\|_{L^2}^2\right)  + 4\|(\sqrt{u})_x\|_{L^2}^2 + 2\|\tilde{v}_x\|_{L^2}^2 
\leqslant C\mathsf{X}(t)\left( 2\mathcal{E}_1 + \|\tilde{v}\|_{L^2}^2\right),
\end{align*}
where we have dropped the non-negative term stemming from logistic growth and 
\begin{align}\label{A24}
\mathsf{X}(t) \triangleq |1-\alpha_1|+|\alpha_2-\alpha_1| + |\alpha_1'|+|\alpha_2'| + |\beta_2-\beta_1| + |\beta_1'| + |\beta_2'|.
\end{align}
Applying Gr\"onwall's inequality and \eqref{BA}, we obtain
\begin{align}\label{A25}
\mathcal{E}_1(t) + \|\tilde{v}(t)\|_{L^2}^2  + \int_0^t \left(\|(\sqrt{u})_x\|_{L^2}^2 +\|\tilde{v}_x\|_{L^2}^2\right)\mathrm{d}\tau \leqslant C.
\end{align}
This completes the entropy-based estimates. \hfill $\square$

\subsubsection{$L^2$-estimates}\label{sect:L2:estimates}

Since $\mathcal{E}_1$ is sub-quadratic with respect to ${u}$ and the spacetime integral of $(\sqrt{u})_x$ is potentially degenerate (due to the lack of a uniform upper bound on $u$ at this point, which precludes obtaining a bound on $u_x$), we next turn to the standard $L^2$-based energy estimates. Letting $\tilde{u}\triangleq u-\alpha$, we obtain from \eqref{1.1a}:
\begin{align}\label{A26}
\tu_t - (\tu\tv)_x - (\tu\beta)_x - (\alpha \tv)_x - (\alpha\beta)_x = \tu_{xx} + u(1-\alpha) - u\tu - \alpha_t.
\end{align}
Taking the $L^2$ inner product of \eqref{A26} with $\tu$ yields
\begin{align}\label{A27}
\frac{\mathrm{d}}{\mathrm{d}t}\|\tu\|_{L^2}^2 + 2\|\tu_x\|_{L^2}^2 + 2\|\tu\sqrt{u}\|_{L^2}^2 = \sum_{k=1}^6 \mathsf{R}_{2,k},
\end{align}
where the $\mathsf{R}_{2,k}$'s are given by 
\begin{align*}
&\mathsf{R}_{2,1} \triangleq 2\int_a^b u(1-\alpha)\tu\,\mathrm{d}x, \quad \mathsf{R}_{2,2} \triangleq 2\int_a^b (\tu\tv)_x\tu\,\mathrm{d}x, \quad \mathsf{R}_{2,3} \triangleq 2\int_a^b (\tu\beta)_x\tu\,\mathrm{d}x, \\
&\mathsf{R}_{2,4} \triangleq 2\int_a^b (\alpha \tv)_x\tu\,\mathrm{d}x, \qquad \ \ \mathsf{R}_{2,5} \triangleq 2\int_a^b(\alpha\beta)_x\tu \,\mathrm{d}x, \ \ \  \mathsf{R}_{2,6} \triangleq 2\int_a^b \alpha_t\tu\,\mathrm{d}x.
\end{align*}
Using the definition and boundedness of $\alpha$, we infer that 
\begin{align}\label{A28}
|\mathsf{R}_{2,1}| \leqslant C\left(|1-\alpha_1| + |\alpha_2-\alpha_1|\right)\left(\|\tu\|_{L^2}^2+1\right).
\end{align}
Integrating by parts and invoking the uniform estimate of $\|\tv\|_{L^2}$ given by \eqref{A25} yield
\begin{align}\label{A29}
|\mathsf{R}_{2,2}| \leqslant 2\|\tu\|_{L^\infty}\|\tv\|_{L^2}\|\tu_x\|_{L^2} \leqslant \frac12\|\tu_x\|_{L^2}^2 + C\|\tu\|_{L^\infty}^2.
\end{align}
Note that for any fixed $t>0$ and $\forall\,x\in[a,b]$,
\begin{align}\label{A30}
\tu(x,t) = u(x,t)-\alpha(x,t) = \int_a^x u_y(y,t)\mathrm{d}y -\frac{x-a}{b-a}[\alpha_2(t) - \alpha_1(t)].
\end{align}
Using \eqref{A30}, \eqref{A16}, \eqref{A25}, and the boundedness of $\alpha_i$, we infer that 
\begin{align}\label{A31}
\|\tu\|_{L^\infty}^2 \leqslant C\left(\|(\sqrt{u})_x\|_{L^2}^2\|u\|_{L^1} + |\alpha_2 - \alpha_1|^2 \right) \leqslant C\left(\|(\sqrt{u})_x\|_{L^2}^2 + |\alpha_2 - \alpha_1|\right).
\end{align}
Substituting \eqref{A31} into \eqref{A29} gives 
\begin{align}\label{A32}
|\mathsf{R}_{2,2}| \leqslant \frac12\|\tu_x\|_{L^2}^2 + C\left(\|(\sqrt{u})_x\|_{L^2}^2 + |\alpha_2 - \alpha_1|\right).
\end{align}
By integration-by-parts, we can show that 
\begin{align}\label{A33}
|\mathsf{R}_{2,3}| + |\mathsf{R}_{2,4}| \leqslant C\left(|\beta_2-\beta_1| \|\tu\|_{L^2}^2 + \|\tv_x\|_{L^2}^2 \right) + \frac12\|\tu_x\|_{L^2}^2
\end{align}
where we applied the boundedness of $\alpha$ and Poincar\'e's inequality for $\tv$. For $\mathsf{R}_{2,5}$ and $\mathsf{R}_{2,6}$, we have
\begin{align}\label{A34}
|\mathsf{R}_{2,5}| + |\mathsf{R}_{2,6}| \leqslant C\left(|\alpha_2-\alpha_1|+|\beta_2-\beta_1|+|\alpha_1'|+|\alpha_2'|\right)\left( \|\tu\|_{L^2}^2 + 1\right).
\end{align}
Substituting \eqref{A28}, \eqref{A32}--\eqref{A34} into \eqref{A27} and applying Schwarz inequality yield
\begin{align}\label{A35}
&\ \frac{\mathrm{d}}{\mathrm{d}t}\|\tu\|_{L^2}^2 + \|\tu_x\|_{L^2}^2 \leqslant C\left(\mathsf{X}(t)+\|(\sqrt{u})_x\|_{L^2}^2+\|\tv_x\|_{L^2}^2\right)\left(\|\tu\|_{L^2}^2 +1\right),
\end{align}
where $\mathsf{X}(t)$ is defined in \eqref{A24} and we have dropped the non-negative term $2\|\tu\sqrt{u}\|_{L^2}^2$ from the left-hand side. Applying Gr\"onwall's inequality, together with \eqref{BA} and \eqref{A25}, we obtain
\begin{align}\label{A36}
\|\tu(t)\|_{L^2}^2 + \int_0^t \|\tu_x(\tau)\|_{L^2}^2 \mathrm{d}\tau \leqslant C.
\end{align}
This completes the $L^2$-estimates. \hfill $\square$

\subsubsection{$H^1$-estimates}

With the $L^2$-estimates at our disposal, we now improve the regularity of the solution. Taking the $L^2$ inner product of \eqref{A26} with $-\tu_{xx}$, we have
\begin{align}\label{A37}
\frac{\mathrm{d}}{\mathrm{d}t}\|\tu_x\|_{L^2}^2 + \|\tu_{xx}\|_{L^2}^2 \leqslant \|(\tu\tv)_x + (\tu\beta)_x + (\alpha\tv)_x + (\alpha\beta)_x + u(1-\alpha) - u\tu - \alpha_t\|_{L^2}^2.
\end{align}
By H\"older, Sobolev and Poincar\'e's inequalities, we infer that $\|(\tu\tv)_x\|_{L^2}^2 \lesssim \|\tu_x\|_{L^2}^2\|\tv_x\|_{L^2}^2$. Using this, together with the boundedness of the boundary data, we can show that 
\begin{align*}
\frac{\mathrm{d}}{\mathrm{d}t}\|\tu_x\|_{L^2}^2 + \|\tu_{xx}\|_{L^2}^2 \leqslant C\left(\mathsf{X}(t) + \|\tu_x\|_{L^2}^2+\|\tv_x\|_{L^2}^2 \right)\left(\|\tu_x\|_{L^2}^2+1\right).
\end{align*}
Applying Gr\"onwall's inequality, together with \eqref{BA}, \eqref{A25} and \eqref{A36}, we obtain 
\begin{align}\label{A39}
\|\tu_x(t)\|_{L^2}^2+ \int_0^t\|\tu_{xx}(\tau)\|_{L^2}^2\mathrm{d}\tau \leqslant C.
\end{align}
Similarly, we can show that
\begin{align*}
\frac{\mathrm{d}}{\mathrm{d}t}\|\tv_x\|_{L^2}^2 + \|\tv_{xx}\|_{L^2}^2 \leqslant C\left(\mathsf{X}(t) + \|\tu_x\|_{L^2}^2+\|\tv_x\|_{L^2}^2 \right)\left(\|\tv_x\|_{L^2}^2+1\right),
\end{align*}
which alongside Gr\"onwall's inequality yields
\begin{align}\label{A41}
\|\tv_x(t)\|_{L^2}^2+ \int_0^t\|\tv_{xx}(\tau)\|_{L^2}^2\mathrm{d}\tau \leqslant C.
\end{align}
This completes the $H^1$-estimates. \hfill $\square$

\subsubsection{Asymptotic stability}

By \eqref{A21}, \eqref{A35} and previous estimates, we can show that 
 \begin{align*}
\left|\left(\|\tu\|_{L^2}^2 + \|\tv\|_{L^2}^2 \right)'(t) \right| \leqslant C\left( \mathsf{X}(t) + \|\tu_x\|_{L^2}^2 +\|\tv_x\|_{L^2}^2 + \|(\sqrt{u})_x\|_{L^2}^2 \right),
\end{align*}
which implies $\left(\|\tu\|_{L^2}^2 + \|\tv\|_{L^2}^2 \right)'(t) \in L^1(\mathbb{R}_+)$. Since $\|\tu_x\|_{L^2}^2 + \|\tv_x\|_{L^2}^2 \in L^1(\mathbb{R}_+)$, it follows from Poincar\'e's inequality that $(\|\tu\|_{L^2}^2 + \|\tv\|_{L^2}^2)(t) \in L^1(\mathbb{R}_+)$. Hence, 
$\left(\|\tu\|_{L^2}^2 + \|\tv\|_{L^2}^2 \right)(t) \in W^{1,1}(\mathbb{R}_+)$. This yields $\left(\|\tu(t)\|_{L^2}^2 + \|\tv(t)\|_{L^2}^2 \right) \to 0$, as $t\to\infty$. The decay of $\|\tu_x(t)\|_{L^2}^2 + \|\tv_x(t)\|_{L^2}^2$ can be established similarly. This completes the proof of the {\it a priori} estimates for the local solution. Theorem \ref{thm1} then follows from a standard continuation argument. We omit the technical details for brevity.

\subsection{Proof of Theorem \ref{thm1} when $\gamma >1$} 

We need the following technical lemma for the proof.

\begin{lemma}
Let $\rho\geqslant 0$ and $s\geqslant 0$ be real numbers. Then the following inequalities hold:
\begin{align}
(\rho-1)^2 &\leqslant \rho^s-1-s(\rho-1), \hspace{1.86 in}  s \geqslant 2; \label{T1} \\
\rho -1 &\leqslant \rho^{s}-1 - s(\rho-1)+(s -1) (1+s^{-1})^\frac{s}{s-1} -s, \quad 1<s\leqslant 2. \label{T2}
\end{align}
\end{lemma}

\begin{proof}
The proof of \eqref{T1} can be found in \cite{ZLMZ}. For \eqref{T2}, consider the function $f(\rho) = \rho^{s}-1-s (\rho-1)+(s -1)\big[ (1+s^{-1})^\frac{s}{s-1}-1 \big] - \rho$ and let $\rho_* = (1+s^{-1})^\frac{1}{s-1}$. It suffices to show $f(\rho)\geqslant 0$ for $\rho\geqslant 0$. Indeed, it is straightforward to verify that $f(\rho_*) = 0$ and $f'(\rho_*) = 0$. Moreover, since $f''(\rho) = s(s-1)\rho^{s-2}>0$ for all $\rho \geqslant 0$ and $1<s\leqslant 2$, it must be true that $f(\rho)\geqslant 0$ for $\rho\geqslant 0$.
\end{proof}

\subsubsection{Entropy-based estimates} We define the expanded entropy functional
\begin{align}\label{re}
\mathcal{E}_2(t) \triangleq \frac{1}{\gamma-1}\int_a^b \left[ u^\gamma - \alpha^\gamma - \gamma \alpha^{\gamma-1}(u-\alpha)\right](x,t) \mathrm{d}x,
\end{align}
where $\alpha$ is given in \eqref{1.4a}. After straightforward calculations, it can be shown that
\begin{align}\label{ee10a}
&\, \frac{\mathrm{d}}{\mathrm{d}t} \left(\mathcal{E}_2 + \frac12\|\tv\|_{L^2}^2 \right) + \gamma \int_a^b u^{\gamma-2}|u_x|^2\,\mathrm{d}x + \|\tv_x\|_{L^2}^2 - \frac{\gamma}{\gamma-1} \int_a^b u(u^{\gamma-1} - \alpha^{\gamma-1})(1-u)\,\mathrm{d}x \notag \\
 =&\, -\gamma(\gamma-2) |\alpha_x|^2 \int_a^b \alpha^{\gamma-3}\tu \,\mathrm{d}x + \gamma |\alpha_x|^2 \int_a^b \alpha^{\gamma-2}\,\mathrm{d}x + \beta_x\|\tv\|_{L^2}^2 - \underline{\gamma \alpha_x \int_a^b \alpha^{\gamma-2} \tu\tv\,\mathrm{d}x} \notag\\
 &\, - \gamma \alpha_x \int_a^b \alpha^{\gamma-2} \beta \tu \,\mathrm{d}x - \gamma \alpha_x\int_a^b \alpha^{\gamma-1} \tv\,\mathrm{d}x - \gamma \alpha_x\int_a^b \alpha^{\gamma-1} \beta \,\mathrm{d}x -\gamma \int_a^b\alpha^{\gamma-2}\alpha_t \tu \,\mathrm{d}x \notag\\
 &\,  + \int_a^b (2\beta\beta_x - \beta_t)\tv\,\mathrm{d}x + \big(\alpha_2^\gamma \beta_2 - \alpha_1^\gamma \beta_1\big) - (\gamma-1) \beta_x \mathcal{E}_2 - \beta_x \int_a^b \big(\alpha^\gamma+\gamma \alpha^{\gamma-1} \tu\big)\,\mathrm{d}x.
\end{align}
To control the integrals on the right-hand side, we first note that by \eqref{T1}--\eqref{T2} and boundedness of $\alpha$,
\begin{align}\label{N1}
\int_a^b |\tu|\mathrm{d}x \leqslant C\left(\mathcal{E}_2 +1\right), \quad \forall\,\gamma>1.
\end{align}
Second, to explore the dissipation mechanism induced by logistic damping, we rewrite the last term on the left-hand side of \eqref{ee10a} as the following:
\begin{align}\label{L2}
&\,- \frac{\gamma}{\gamma-1} \int_a^b u(u^{\gamma-1} - \alpha^{\gamma-1})(1-u)\,\mathrm{d}x = \frac{\gamma}{\gamma-1} \int_a^b u(u^{\gamma-1} - \alpha^{\gamma-1})\tu\,\mathrm{d}x \notag\\
&\,\qquad \qquad \qquad -\frac{\gamma}{\gamma-1} \int_a^b \big[u^\gamma-\alpha^\gamma-\gamma\alpha^{\gamma-1}\tu\big](1-\alpha)\,\mathrm{d}x - \gamma \int_a^b \alpha^{\gamma-1}(1-\alpha)\tu\,\mathrm{d}x,
\end{align}
Then, using \eqref{N1}, \eqref{L2} and Schwarz inequality, we can upgrade \eqref{ee10a} as 
\begin{align}\label{N2}
&\, \frac{\mathrm{d}}{\mathrm{d}t} \left(\mathcal{E}_2 + \frac12\|\tv\|_{L^2}^2 \right) + \gamma \int_a^b u^{\gamma-2}|u_x|^2\,\mathrm{d}x + \|\tv_x\|_{L^2}^2 + \frac{\gamma}{\gamma-1} \int_a^b u(u^{\gamma-1} - \alpha^{\gamma-1})\tu\,\mathrm{d}x \notag \\
 \leqslant &\, C\mathsf{X}(t)\left(\mathcal{E}_2 + \|\tv\|_{L^2}^2 +1\right) - \gamma \alpha_x \int_a^b \alpha^{\gamma-2} \tu\tv\,\mathrm{d}x,
\end{align}
where $\mathsf{X}(t)$ is defined in \eqref{A24}. Next, we use the domain-splitting argument in $\S$2.1.1 to regulate the nonlinearity on the right-hand side of \eqref{N2}.

For fixed $t>0$, by virtue of the mean value theorem and positivity of $u$ and $\alpha$, it holds that
\begin{align}\label{L9}
\frac{\gamma}{\gamma-1} \int_a^b u(u^{\gamma-1} - \alpha^{\gamma-1})\tu \,\mathrm{d}x \geqslant \gamma \int_{\{u>\alpha\}} u\,\xi^{\gamma-2} \tu^2 \,\mathrm{d}x,
\end{align}
where $\xi$ is between $\alpha$ and $u$. For any $\gamma > 1$, we can show that 
\begin{align}\label{L11}
\gamma \int_{\{u>\alpha\}} u\,\xi^{\gamma-2} \tu^2 \,\mathrm{d}x \geqslant \gamma \int_{\{u>\alpha\}} \alpha^{\gamma-1} \tu^2 \,\mathrm{d}x \geqslant \gamma\underline{\alpha}^{\gamma-1}\int_{\{u>\alpha\}} \tu^2 \,\mathrm{d}x,
\end{align}
where $\underline{\alpha}>0$ denotes the lower bound for $\alpha$.  Combining \eqref{L9} and \eqref{L11}, we obtain
\begin{align}\label{L12}
\frac{\gamma}{\gamma-1} \int_a^b u(u^{\gamma-1} - \alpha^{\gamma-1})\tu \,\mathrm{d}x \geqslant \gamma\underline{\alpha}^{\gamma-1}\int_{\{u>\alpha\}} \tu^2 \,\mathrm{d}x.
\end{align}
For the last term on the right-hand side of \eqref{N2}, note that 
\begin{align}\label{L14}
\Big|\gamma\alpha_x\int_a^b \alpha^{\gamma-2}\tu\tv\,\mathrm{d}x \Big| 
&\leqslant \left\{
\begin{aligned}
&\gamma |\alpha_x| \underline{\alpha}^{\gamma-2} \int_{\{u>\alpha\}} |\tu\tv| \,\mathrm{d}x + 2\gamma |\alpha_x| \int_{\{u\leqslant \alpha\}} \alpha^{\gamma-1}|\tv|\,\mathrm{d}x,\quad \gamma\leqslant 2,\\
&\gamma |\alpha_x| \overline{\alpha}^{\gamma-2} \int_{\{u>\alpha\}} |\tu\tv| \,\mathrm{d}x + 2\gamma |\alpha_x| \int_{\{u\leqslant \alpha\}} \alpha^{\gamma-1}|\tv|\,\mathrm{d}x,\quad \gamma>2,
\end{aligned}
\right.
\end{align}
where $\overline{\alpha}>0$ denotes the upper bound for $\alpha$. When $\gamma\leqslant 2$, we have 
\begin{align}\label{L15}
\gamma |\alpha_x| \underline{\alpha}^{\gamma-2} \int_{\{u>\alpha\}} |\tu\tv| \,\mathrm{d}x 
\leqslant \frac{\gamma\underline{\alpha}^{\gamma-1}}{2}\int_{\{u>\alpha\}} \tu^2 \,\mathrm{d}x + \frac{\gamma|\alpha_x|^2\underline{\alpha}^{\gamma-3}}{2}\int_{\{u>\alpha\}} \tv^2 \,\mathrm{d}x,
\end{align}
and a similar estimate holds for $\gamma>2$. Moreover, 
\begin{align}\label{L16}
2\gamma |\alpha_x| \int_{\{u\leqslant \alpha\}} \alpha^{\gamma-1}|\tv|\,\mathrm{d}x \leqslant \gamma |\alpha_x| \overline{\alpha}^{\gamma-1} \int_{\{u\leqslant \alpha\}} (|\tv|^2+1)\,\mathrm{d}x.
\end{align}
Substituting \eqref{L15}--\eqref{L16} into \eqref{L14} and using the boundedness of $\alpha_i$ yield
\begin{align}\label{L17}
\Big|\gamma\alpha_x\int_a^b \alpha^{\gamma-2}\tu\tv\,\mathrm{d}x \Big| \leqslant \frac{\gamma\underline{\alpha}^{\gamma-1}}{2}\int_{\{u>\alpha\}} \tu^2 \,\mathrm{d}x + C|\alpha_2-\alpha_1| \big(\|\tv\|_{L^2}^2+1\big).
\end{align}
The combination of \eqref{N2}, \eqref{L12} and \eqref{L17} gives 
\begin{align}\label{N3}
\frac{\mathrm{d}}{\mathrm{d}t} \left(\mathcal{E}_2 + \frac12\|\tv\|_{L^2}^2 \right) + \gamma \int_a^b u^{\gamma-2}|u_x|^2\,\mathrm{d}x + \|\tv_x\|_{L^2}^2 \leqslant C\mathsf{X}(t)\left(\mathcal{E}_2 + \|\tv\|_{L^2}^2 +1\right),
\end{align}
where we have dropped the non-negative term involving the spatial integral of $\tu^2$ from the left-hand side. Applying Gr\"onwall's inequality and \eqref{BA}, we obtain  
\begin{align}\label{N4}
\mathcal{E}_2(t) + \|\tv(t)\|_{L^2}^2 + \int_0^t \left(\int_a^b u^{\gamma-2}|u_x|^2\,\mathrm{d}x + \|\tv_x\|_{L^2}^2 \right)(\tau)\mathrm{d}\tau \leqslant C.
\end{align}
This completes the entropy-based estimates. \hfill $\square$

\subsubsection{$L^2$-estimates}

Writing \eqref{1.1a} and \eqref{1.1b} in terms of the perturbation, we have 
\begin{subequations}\label{W1}
\begin{alignat}{2}
\tu_t - [(\tu+\alpha)\tv]_x &= \tu_{xx} + (\beta \tu)_x + \alpha_x\beta + \alpha\beta_x - \alpha_t + (\tu+\alpha)(1-\alpha-\tu), \label{W1a}\\
\tv_t - [(\tu+\alpha)^\gamma]_x &= \tv_{xx} + 2\tv\tv_x + 2\tv\beta_x + 2\beta \tv_x + 2\beta\beta_x - \beta_t. \label{W1b}
\end{alignat}
\end{subequations}
Taking the $L^2$ inner product of \eqref{W1a} with $\tu$, we can show that 
\begin{align}\label{W2}
\frac{1}{2}\frac{\mathrm{d}}{\mathrm{d}t} \|\tu\|_{L^2}^2 + \|\tu_x\|_{L^2}^2 \leqslant \int_a^b (\tu+\alpha)\tv\tu_x \mathrm{d}x + C\mathsf{X}(t)\left(\|\tu\|_{L^2}^2+1\right).
\end{align}
Using Poincar\'e's inequality, we infer that 
\begin{align*}
\Big|\int_a^b \tu \tv\tu_x\dx\Big| \leqslant \frac{1}{4}\|\tu_x\|_{L^2}^2 + \|\tu\|_{L^2}^2\|\tv\|_{L^\infty}^2 \leqslant \frac{1}{4}\|\tu_x\|_{L^2}^2 + C\|\tu\|_{L^2}^2\|\tv_x\|_{L^2}^2.
\end{align*}
Similarly, by the boundedness of $\alpha$, we have 
\begin{align*}
\Big|\int_a^b \alpha \tv\tu_x\dx\Big| \leqslant \frac{1}{4}\|\tu_x\|_{L^2}^2 + C\|\tv_x\|_{L^2}^2.
\end{align*}
Then we update \eqref{W2} as 
\begin{align*}
\frac{\mathrm{d}}{\mathrm{d}t} \|\tu\|_{L^2}^2 + \|\tu_x\|_{L^2}^2 \leqslant  C\left(\mathsf{X}(t)+ \|\tv_x\|_{L^2}^2\right) \left(\|\tu\|_{L^2}^2+1\right).
\end{align*}
Applying Gr\"onwall's inequality, \eqref{BA}, \eqref{N4}, and the boundedness of $\alpha$, we get
\begin{align*}
\|\tu(t)\|_{L^2}^2 + \int_0^t \|\tu_x(\tau)\|_{L^2}^2 \mathrm{d}\tau \leqslant C.
\end{align*}
This completes the $L^2$-estimates. \hfill $\square$

\vspace{.1 in}

The remaining estimates are similar to those in $\S$2.1.3--$\S$2.1.4. The only difference between the proofs is the estimate of $\|(u^\gamma)_x\|_{L^2}^2$. Note that after \eqref{A39} is established, Sobolev and Poincar\'e's inequalities imply $\|\tu(t)\|_{L^\infty} \leqslant C$. This alongside the boundedness of $\alpha$ yields $\|u(t)\|_{L^\infty} \leqslant C$. Then it follows that $\|(u^\gamma)_x\|_{L^2}^2 \leqslant C\|u_x\|_{L^2}^2$, by which one can establish \eqref{A41}.

\section{Parabolic-Hyperbolic System under Dynamic Dirichlet B.C.}\label{PH}

This section addresses the long-time dynamics of the hyperbolic counterpart of system \eqref{0.1} under dynamic boundary conditions. We consider the following initial-boundary value problem:
\begin{subequations}\label{2.1}
\begin{alignat}{4}
u_t - (u v)_x &= u_{xx} + u(1-u), \quad &x&\in (a,b),\ t>0, \label{2.1a}\\
v_t - (u^\gamma)_x &=  0, \quad &x&\in (a,b),\ t>0; \label{2.1b}\\
(u,v)(x,0)&=(u_0,v_0)(x), \quad &x&\in [a,b], \label{2.2}\\
u(a,t)&=\alpha_1(t),\quad u(b,t)=\alpha_2(t), \quad &t& \geqslant 0. \label{2.3}
\end{alignat}
\end{subequations}
To determine an appropriate ansatz for $v$, we integrate \eqref{2.1b} with respect to $x$ and $t$, yielding:
\begin{align*}
\int_a^b v(x,t)\mathrm{d}x = \int_a^b v_0(x)\mathrm{d}x + \int_0^t \left(\alpha_2^\gamma - \alpha_1^\gamma \right)(\tau)\mathrm{d}\tau.
\end{align*}
Define $\Psi(t)$ as the dynamic spatial average of $v(x,t)$:
\begin{align}\label{2.5}
\Psi(t) \triangleq \frac{1}{b-a}\int_a^b v(x,t)\mathrm{d}x = \frac{1}{b-a}\left( \int_a^b v_0(x)\mathrm{d}x + \int_0^t \left(\alpha_2^\gamma - \alpha_1^\gamma \right)(\tau)\mathrm{d}\tau \right).
\end{align}
The fluctuation $\tv(x,t) \triangleq v(x,t) - \Psi(t)$ then satisfies the zero-mass condition for all time, and consequently, the Poincar\'e inequality applies. This indicates that $\Psi(t)$ provides the correct ansatz for the behavior of $v(x,t)$. Our main results are summarized in the following theorem.

\begin{theorem}\label{thm2}
Consider the initial-boundary value problem \eqref{2.1} with $\gamma=1$ or $\gamma\geqslant 2$. Suppose that the initial data satisfy $u_0>0$, $u_0,v_0 \in H^1((a,b))$, and are compatible with the boundary conditions. Assume that the boundary data $\alpha_1$, $\alpha_2$ are smooth functions on $\mathbb{R}_+$, satisfying 
\begin{align}
\alpha_1(t) \geqslant  \underline{\alpha_1}, \quad \alpha_2(t) \geqslant  \underline{\alpha_2}, \quad \text{and} \quad \alpha_1-1,\alpha_2-1 \in W^{1,1}(\mathbb{R}_+),\label{BA-new}
\end{align}
where $\underline{\alpha_1}$, $\underline{\alpha_2}>0$ are constants. Then there exists a unique solution to the IBVP, such that for $\forall\,t>0$,
\begin{align*}
\|(u-\alpha)(t)\|^2_{H^1}+\|(v-\Psi)(t)\|^2_{H^1} + \int_0^t \left(\|(u-\alpha)(\tau)\|^2_{H^2}+\|(v-\Psi)(\tau)\|^2_{H^1}\right)\mathrm{d}\tau \leqslant C,
\end{align*}
where $\alpha$ and $\Psi$ are defined by \eqref{1.4a} and \eqref{2.5}, respectively. Moreover, the solution satisfies
\begin{align*}
\|(u-\alpha)(t)\|_{H^1}+\|(v-\Psi)(t)\|_{H^1} \to 0\quad\text{as}\quad t\to\infty.
\end{align*}
\end{theorem}

\begin{remark}
The global dynamics of \eqref{2.1} with $1<\gamma<2$ remain unexplored. Preliminary analysis suggests mathematical challenges with no current resolution. This is left for future work.
\end{remark}

Due to the lack of diffusion in $v$, the proof of Theorem \ref{thm2} differs significantly from that of Theorem \ref{thm2}. A key difference lies in establishing the time integrability for $\tv_x$. We accomplish this by deriving a damped equation for $\tv_x$. Again, we separate the proof into two subsections for $\gamma=1$ and $\gamma \geqslant 2$.

\subsection{Proof of Theorem \ref{thm2} when $\gamma=1$}

To begin, note that \eqref{A17} is still valid for the IBVP \eqref{2.1}. To compensate the quadratic nonlinearity on the left-hand side of \eqref{A17}, we take the $L^2$ inner product of \eqref{2.1b} with $v$ and then add the result to \eqref{A17} to get
\begin{align}\label{2.8}
\frac{\mathrm{d}}{\mathrm{d}t} \left(2\mathcal{E}_1 + \|v\|_{L^2}^2\right) + 4\|(\sqrt{u})_x\|_{L^2}^2 \leqslant C\mathsf{Y}(t)\left(\mathcal{E}_1 + \|v\|_{L^2}^2+1\right),
\end{align}
where we have dropped the non-negative term stemming from logistic growth and 
\begin{align}\label{Y}
\mathsf{Y}(t) \triangleq |1-\alpha_1| + |\alpha_2-\alpha_1| + |\alpha_1'| + |\alpha_2'|.
\end{align}
Then it follows from \eqref{2.8} and the boundedness of the boundary data that 
\begin{align*}
\mathcal{E}_1(t) + \|\tv(t)\|_{L^2}^2  + \int_0^t \|(\sqrt{u})_x(\tau)\|_{L^2}^2 \mathrm{d}\tau \leqslant C.
\end{align*}
This completes the entropy-based estimates.

For $L^2$-estimates, similar to \eqref{A26}, we have
\begin{align}\label{2.10}
\tu_t - (\tu\tv)_x - (\tu\Psi)_x - (\alpha \tv)_x - (\alpha\Psi)_x = \tu_{xx} + u(1-\alpha) - u\tu - \alpha_t.
\end{align}
Following the arguments in $\S$2.1.2 and noticing $\Psi_x=0$, we can show that 
\begin{align}\label{2.11}
\frac{\mathrm{d}}{\mathrm{d}t}\|\tu\|_{L^2}^2 + \|\tu_x\|_{L^2}^2 + 2\int_a^b \alpha \tv\tu_x\mathrm{d}x \leqslant C\left( \mathsf{Y}(t) + \|(\sqrt{u})_x\|_{L^2}^2\right)\left( \|\tu\|_{L^2}^2 +1\right).
\end{align}
Writing \eqref{2.1b} in terms of the perturbation, we have
\begin{align}\label{2.12}
\tv_t - \tu_x = -\Psi'(t) + \alpha_x =0.
\end{align}
Taking the $L^2$ inner product of \eqref{2.12} with $2\alpha \tv$, then adding the result to \eqref{2.11}, we obtain
\begin{align}\label{2.13}
\frac{\mathrm{d}}{\mathrm{d}t} \left(\|\tu\|_{L^2}^2 +\|\sqrt{\alpha}\,\tv\|_{L^2}^2\right) + \|\tu_x\|_{L^2}^2 \leqslant C\left( \mathsf{Y}(t) + \|(\sqrt{u})_x\|_{L^2}^2 \right)\left( \|\tu\|_{L^2}^2 + \|\tv\|_{L^2}^2 +1\right).
\end{align}
Since $\alpha$ is bounded away from zero, Gr\"onwall's inequality then yields
\begin{align}\label{2.14}
\|\tu(t)\|_{L^2}^2 +\|\tv(t)\|_{L^2}^2+ \int_0^t \|\tu_x(\tau)\|_{L^2}^2\mathrm{d}\tau \leqslant C.
\end{align}
This completes the $L^2$-estimates. Next, we establish the temporal integrability of $\|\tv_x\|_{L^2}^2$. 

Substituting $\tu_{xx} = \tv_{xt}$ into \eqref{2.10}, we obtain
\begin{align}\label{2.15}
\tu_t - (\tu\tv)_x - (\tu\Psi)_x - (\alpha \tv)_x - (\alpha\Psi)_x = \tv_{xt} + u(1-\alpha) - u\tu - \alpha_t.
\end{align}
Taking the $L^2$ inner product of \eqref{2.15} with $\tv_x$, we have
\begin{align}\label{2.16}
\frac12\frac{\mathrm{d}}{\mathrm{d}t}\|\tv_x\|_{L^2}^2 + \|\sqrt{\alpha}\,\tv_x\|_{L^2}^2 = \sum_{k=1}^{10} \mathsf{R}_{3,k}.
\end{align}
where the quantities on the right-hand side are given by
\begin{align*}
&\mathsf{R}_{3,1} \triangleq -\int_a^b \alpha_x\tv\tv_x\mathrm{d}x, \quad \mathsf{R}_{3,2} \triangleq - \int_a^b (\tu\tv)_x\tv_x\mathrm{d}x, \quad \mathsf{R}_{3,3} \triangleq - \int_a^b (\tu\Psi)_x\tv_x\mathrm{d}x, \\
&\mathsf{R}_{3,4} \triangleq - \int_a^b(\alpha\Psi)_x\tv_x\mathrm{d}x, \quad \mathsf{R}_{3,5} \triangleq - \int_a^b \tu(1-\alpha)\tv_x\mathrm{d}x, \quad 
\mathsf{R}_{3,6} \triangleq - \int_a^b \alpha(1-\alpha)\tv_x\mathrm{d}x,\\ &\mathsf{R}_{3,7} \triangleq \int_a^b \tu^2\tv_x\mathrm{d}x, \quad 
\mathsf{R}_{3,8} \triangleq \int_a^b \alpha\tu\tv_x\mathrm{d}x, \quad \mathsf{R}_{3,9} \triangleq \int_a^b \alpha_t\tv_x\mathrm{d}x, \quad \mathsf{R}_{3,10} \triangleq \int_a^b \tu_t\tv_x\mathrm{d}x.
\end{align*}
Using the boundedness of $\alpha$ and \eqref{2.14}, we obtain
\begin{align}\label{2.18}
|\mathsf{R}_{3,1}| \leqslant \frac{1}{18}\|\sqrt{\alpha}\,\tv_x\|_{L^2}^2 + C|\alpha_2-\alpha_1|.
\end{align}
Since $\tv$ satisfies Poincar\'e's inequality, we can show that 
\begin{align*}
|\mathsf{R}_{3,2}| \leqslant \frac{1}{18}\|\sqrt{\alpha}\,\tv_x\|_{L^2}^2 + C\|\tu_x\|^2\|\tv_x\|_{L^2}^2.
\end{align*}
It follows from the definition of $\Psi$ and its boundedness,
\begin{align}\label{2.20}
|\mathsf{R}_{3,3}| + |\mathsf{R}_{3,4}| \leqslant \frac{1}{9}\|\sqrt{\alpha}\,\tv_x\|_{L^2}^2 + C\|\tu_x\|_{L^2}^2.
\end{align}
The boundedness of $\alpha$ and Poincar\'e's inequality imply 
\begin{align}\label{2.22}
|\mathsf{R}_{3,5}| + |\mathsf{R}_{3,6}| \leqslant \frac{1}{9}\|\sqrt{\alpha}\,\tv_x\|_{L^2}^2 + C\|\tu_x\|_{L^2}^2.
\end{align}
Since $\|\tu\|_{L^2}$ and $\alpha$ are uniformly bounded, by Poincar\'e's inequality, 
\begin{align}\label{2.24}
|\mathsf{R}_{3,7}| + |\mathsf{R}_{3,8}| \leqslant \frac{1}{9}\|\sqrt{\alpha}\,\tv_x\|_{L^2}^2 + C\|\tu_x\|_{L^2}^2.
\end{align}
For $\mathsf{R}_{3,9}$, it is obvious that 
\begin{align}\label{2.26}
|\mathsf{R}_{3,9}| \leqslant \frac{1}{18}\|\sqrt{\alpha}\,\tv_x\|_{L^2}^2 + C\left(|\alpha_1'| + |\alpha_2'|\right).
\end{align}
Lastly, by switching time derivative, applying \eqref{2.12}, and integrating by parts, we obtain
\begin{align}\label{2.27}
\mathsf{R}_{3,10} = \frac{\mathrm{d}}{\mathrm{d}t}\left(\int_a^b \tu\tv_x\mathrm{d}x\right) + \|\tu_x\|_{L^2}^2.
\end{align}
Substituting \eqref{2.18}-\eqref{2.27} into \eqref{2.16} yields
\begin{align}\label{2.28}
\frac{\mathrm{d}}{\mathrm{d}t}\left(\|\tv_x\|_{L^2}^2 - 2\int_a^b \tu\tv_x\mathrm{d}x\right) + \|\sqrt{\alpha}\,\tv_x\|_{L^2}^2 \leqslant C\left(\mathsf{Y}(t)+\|\tu_x\|_{L^2}^2\right)\left(\|\tv_x\|_{L^2}^2 + 1\right).
\end{align}
To close the energy estimate, we note that \eqref{2.14} upgrades \eqref{2.13} as 
\begin{align}\label{2.29}
\frac{\mathrm{d}}{\mathrm{d}t} \left(\|\tu\|_{L^2}^2 +\|\sqrt{\alpha}\,\tv\|_{L^2}^2\right) + \|\tu_x\|_{L^2}^2 \leqslant C\left( \mathsf{Y}(t) + \|(\sqrt{u})_x\|_{L^2}^2 \right).
\end{align}
Multiplying \eqref{2.29} by 2, adding the result to \eqref{2.28}, and applying Gr\"onwall's inequality, we obtain
\begin{align}\label{2.32}
\|\tv_x(t)\|_{L^2}^2 + \int_0^t \|\tv_x(\tau)\|_{L^2}^2\mathrm{d}\tau \leqslant C.
\end{align}
As a consequence of \eqref{2.32}, by repeating the arguments in $\S$2.1.3, we can show that 
\begin{align*}
\|\tu_x(t)\|_{L^2}^2 + \int_0^t \|\tu_{xx}(\tau)\|_{L^2}^2\mathrm{d}\tau \leqslant C.
\end{align*}
This completes the $H^1$-estimates.

\subsection{Proof of Theorem \ref{thm2} when $\gamma\geqslant 2$} 

This subsection constitutes the most technically demanding portion of the paper, due to the absence of diffusion in $v$, the strong nonlinearity introduced by the term $(u^\gamma)_x$, and the deeply coupled nature of the unknown functions. To overcome these challenges, we refine the separation-compensation scheme established in prior work \cite{ZLMZ}, a process that relies extensively on the systematic application of elementary inequalities. First, we recall the following inequalities:

\begin{lemma}[\cite{ZLMZ}]
Let $\rho\geqslant 0$ and $s\geqslant 0$ be real numbers. Then the following inequalities hold:
\begin{align}
|\rho-1|^s &\leqslant \rho^s - 1 - s(\rho-1), \quad s \geqslant 2; \label{T3}\\
\rho^s -1 &\geqslant s(\rho-1), \hspace{.75 in} s \geqslant 1; \label{T4} \\
|\rho^s -1| &\leqslant |\rho-1|, \hspace{.85 in} s\leqslant 1. \label{T5}
\end{align}
\end{lemma}

We begin by repeating the arguments in $\S$2.2.1 to get \eqref{N3} without the diffusion of $v$:
\begin{align}\label{X1}
\frac{\mathrm{d}}{\mathrm{d}t} \left(2\mathcal{E}_2 + \|\tv\|_{L^2}^2 \right) + 2\gamma \int_a^b u^{\gamma-2}|u_x|^2\,\mathrm{d}x \leqslant C\mathsf{Y}(t)\left(\mathcal{E}_2 + \|\tv\|_{L^2}^2 +1\right),
\end{align}
where $\mathsf{Y}(t)$ is defined in \eqref{Y}. Gr\"onwall's inequality implies
\begin{align}\label{X1a}
\mathcal{E}_2(t) + \|\tv(t)\|_{L^2}^2 + \int_0^t \int_a^b u^{\gamma-2}|u_x|^2\,\mathrm{d}x\mathrm{d}\tau \leqslant C.
\end{align}
As a consequence of \eqref{re}, \eqref{X1a} and \eqref{T3}, we obtain
\begin{align}\label{X1b}
\|\tu(t)\|_{L^\gamma}\leqslant C. 
\end{align}
With \eqref{X1a}, we can upgrade \eqref{X1} as 
\begin{align*}
\frac{\mathrm{d}}{\mathrm{d}t} \left(2\mathcal{E}_2 + \|\tv\|_{L^2}^2 \right) + 2\gamma \int_a^b u^{\gamma-2}|u_x|^2\,\mathrm{d}x \leqslant C\mathsf{Y}(t).
\end{align*}
Since $u_x=\tu_x+\alpha_x$ and $u\geqslant 0$, we know $2u^{\gamma-2}|u_x|^2 \geqslant u^{\gamma-2}|\tu_x|^2-2u^{\gamma-2}|\alpha_x|^2$. This implies
\begin{align}\label{X3}
\frac{\mathrm{d}}{\mathrm{d}t} \left(2\mathcal{E}_2 + \|\tv\|_{L^2}^2 \right) + \gamma \int_a^b u^{\gamma-2}|\tu_x|^2\,\mathrm{d}x \leqslant C\mathsf{Y}(t) + 2\gamma |\alpha_x|^2 \int_a^b u^{\gamma-2}\,\mathrm{d}x.
\end{align}
Since $\gamma\geqslant 2$ and $u\geqslant 0$, by Young's inequality, \eqref{re} and boundedness of $\alpha$, we infer that 
\begin{align}\label{X4}
\int_a^b u^{\gamma-2} \,\mathrm{d}x \leqslant \frac{\gamma-2}{\gamma}\int_a^b u^\gamma \,\mathrm{d}x + \frac{2(b-a)}{\gamma} \leqslant C\left(\mathcal{E}_2+1\right) \leqslant C,
\end{align}
where we also applied the boundedness of $\mathcal{E}_2$. By \eqref{X4}, we update \eqref{X3} as
\begin{align}\label{X5}
\frac{\mathrm{d}}{\mathrm{d}t} \left(2\mathcal{E}_2 + \|\tv\|_{L^2}^2 \right) + \gamma \int_a^b u^{\gamma-2}|\tu_x|^2\,\mathrm{d}x \leqslant C\mathsf{Y}(t).
\end{align}
which yields the upgraded version of \eqref{X1a}:
\begin{align}\label{X5a}
\mathcal{E}_2(t) + \|\tv(t)\|_{L^2}^2 + \int_0^t \int_a^b u^{\gamma-2}|\tu_x|^2\,\mathrm{d}x\mathrm{d}\tau \leqslant C.
\end{align}
Next, we recover the regular dissipation for $\tu$ from \eqref{X5}.

\subsubsection{Recovery of regular dissipation for $\tu$}

\

\vspace{.05 in}

{\bf Step 1 (separation).} We first write the dissipation term on the left-hand side of \eqref{X5} as 
\begin{align}\label{ee26}
\gamma \int_a^b u^{\gamma-2} |\tu_x|^2\,\mathrm{d}x = \gamma \int_a^b \left[(\tu+\alpha)^{\gamma-2} - \alpha^{\gamma-2}\right] |\tu_x|^2\,\mathrm{d}x + \gamma \int_a^b \alpha^{\gamma-2} |\tu_x|^2\,\mathrm{d}x.
\end{align}
\underline{When $2\leqslant \gamma \ls 3$}, the first term on the right of \eqref{ee26} satisfies, by \eqref{T5},
\begin{align*}
\big|(\tu+\alpha)^{\gamma-2}-\alpha^{\gamma-2}\big| \ls
\alpha^{\gamma-3}|\tu|\ls \frac{\alpha^{\gamma-4}\tu^2}{2}+\frac{\alpha^{\gamma-2}}{2},
\end{align*}
which implies 
\begin{align}\label{ee28}
(\tu+\alpha)^{\gamma-2}-\alpha^{\gamma-2} \gs  -\frac{\alpha^{\gamma-4}\tu^2}{2}-\frac{\alpha^{\gamma-2}}{2}.
\end{align}
Using \eqref{ee28} and the boundedness of $\alpha$, we update \eqref{ee26} as 
\begin{align}\label{ee29}
\gamma \int_a^b u^{\gamma-2} |\tu_x|^2\,\mathrm{d}x &\gs  -\frac{\gamma}{2} \underline{\alpha}^{\gamma-4} \|\tu\tu_x\|_{L^2}^2 + \frac{\gamma}{2} \underline{\alpha}^{\gamma-2}\|\tu_x\|_{L^2}^2.
\end{align}
Substituting \eqref{ee29} into \eqref{X5}, we obtain
\begin{align}\label{ee30}
\frac{\mathrm{d}}{\mathrm{d}t} \left(2\mathcal{E}_2 + \|\tv\|_{L^2}^2 \right) + \frac{\gamma}{2} \underline{\alpha}^{\gamma-2}\|\tu_x\|_{L^2}^2 - \frac{\gamma}{2} \underline{\alpha}^{\gamma-4} \|\tu\tu_x\|_{L^2}^2 \ls C\mathsf{Y}(t).
\end{align}
\underline{When $3<\gamma\ls 4$}, by \eqref{T4}, we have
\begin{align*}
(\tu+\alpha)^{\gamma-2} -\alpha^{\gamma-2} \gs   (\gamma-2)\alpha^{\gamma-3}\tu,
\end{align*}
which upgrades \eqref{X5} as 
\begin{align}\label{ee32}
\frac{\mathrm{d}}{\mathrm{d}t} \left(2\mathcal{E}_2 + \|\tv\|_{L^2}^2 \right) + \gamma \int_a^b \alpha^{\gamma-2} |\tu_x|^2\,\dx + \gamma(\gamma-2) \int_a^b \alpha^{\gamma-3}\tu|\tu_x|^2\,\mathrm{d}x 
 \ls C\mathsf{Y}(t).
\end{align}
Since $\alpha\gs  \underline{\alpha}$ and $2(\gamma-2)\alpha^{\gamma-3}\tu \gs  - \alpha^{\gamma-2} - (\gamma-2)^2 \alpha^{\gamma-4}\tu^2$, we have
\begin{align*}
\gamma(\gamma-2) \int_a^b \alpha^{\gamma-3}\tu|\tu_x|^2\,\mathrm{d}x \gs  -\frac{\gamma}{2} \int_a^b \alpha^{\gamma-2}|\tu_x|^2\,\mathrm{d}x - \frac{\gamma(\gamma-2)^2}{2} \underline{\alpha}^{\gamma-4} \|\tu\tu_x\|_{L^2}^2,
\end{align*}
which updates \eqref{ee32} as 
\begin{align}\label{ee35}
\frac{\mathrm{d}}{\mathrm{d}t} \left(2\mathcal{E}_2 + \|\tv\|_{L^2}^2 \right) + \frac{\gamma}{2}\underline{\alpha}^{\gamma-2} \|\tu_x\|_{L^2}^2 - \frac{\gamma(\gamma-2)^2}{2} \underline{\alpha}^{\gamma-4} \|\tu\tu_x\|_{L^2}^2 \ls C\mathsf{Y}(t).
\end{align}
\underline{When $\gamma>4$}, we still have \eqref{ee32}. By Young's inequality, we can show that
\begin{align}\label{ee36}
(\gamma-2) \int_a^b \alpha^{\gamma-3}\tu|\tu_x|^2\,\mathrm{d}x \gs  - \frac12\int_a^b\alpha^{\gamma-2}|\tu_x|^2\,\mathrm{d}x - \mathsf{d}_1\int_a^b |\tu|^{\gamma-1}|\tu_x|^2\,\mathrm{d}x,
\end{align}
where $\mathsf{d}_1=\frac{2^{\gamma-2}\gamma(\gamma-2)^{2\gamma-3}}{(\gamma-1)^{\gamma-1}\underline{\alpha}}$. Substituting \eqref{ee36} into \eqref{ee32}, we obtain
\begin{align}\label{ee37}
\frac{\mathrm{d}}{\mathrm{d}t} \left(2\mathcal{E}_2 + \|\tv\|_{L^2}^2 \right) + \frac{\gamma}{2} \underline{\alpha}^{\gamma-2} \|\tu_x\|_{L^2}^2 - \mathsf{d}_1\int_a^b |\tu|^{\gamma-1}|\tu_x|^2\,\mathrm{d}x \ls C\mathsf{Y}(t).
 \end{align}

\vspace{.05 in}

{\bf Step 2 (compensation).} We carry out $L^p$-based energy estimates to compensate the nonlinearity. 

\vspace{.05 in}

\noindent \underline{When $2\leqslant \gamma\ls 4$}, taking the $L^2$ inner product of \eqref{2.10} with $4\tu^3$, we obtain
\begin{align}\label{ee39}
\frac{\mathrm{d}}{\mathrm{d}t} \|\tu\|_{L^4}^4 + 12 \|\tu\tu_x\|_{L^2}^2 + 4\|\sqrt{u}\,\tu^2\|_{L^2}^2= \mathsf{R}_{4,1} + \mathsf{R}_{4,2},
\end{align}
where the two quantities on the right-hand side are defined as
\begin{align*}
\mathsf{R}_{4,1} \triangleq -12 \int_a^b \tu^2\tu_x u \tv\,\mathrm{d}x, \quad \mathsf{R}_{4,2} \triangleq 4\int_a^b \left[\Psi\alpha_x - \alpha_t + u (1-\alpha)\right]\tu^3\,\mathrm{d}x.
\end{align*}
To estimate $\mathsf{R}_{4,1}$, we note that by H\"older's inequality,
\begin{align}\label{ee40}
|\mathsf{R}_{4,1}| \ls 12\Big(\int_a^b u^{\gamma-2}|\tu_x|^2\,\mathrm{d}x \Big)^\frac12 \Big( \int_a^b u^{4-\gamma} \tu^4\tv^2\,\mathrm{d}x\Big)^\frac12.
\end{align}
Since $2\leqslant \gamma\ls 4$ and $u \gs 0$, by Young's inequality, we have
\begin{align}\label{ee42}
\int_a^b u^{4-\gamma} \tu^4\tv^2\,\mathrm{d}x \ls (4-\gamma)\int_a^b (\tu^2+\alpha^2)\tu^4\tv^2\,\mathrm{d}x + \frac{\gamma-2}{2}\int_a^b \tu^4\tv^2\,\mathrm{d}x.
\end{align}
Using the boundedness of $\alpha$ and $\|\tv\|_{L^2}$, we have
\begin{align*}
\int_a^b u^{4-\gamma} \tu^4\tv^2\,\mathrm{d}x \ls C\left(\|\tu\|_{L^\infty}^6 + \|\tu\|_{L^\infty}^4\right).
\end{align*}
Since $\tu(a,t)=0$, we infer that $|\tu(x,t)|^3 \ls 3\|\tu\|_{L^2}\|\tu\tu_x\|_{L^2} \ls C\|\tu\tu_x\|_{L^2}$, where we applied the uniform boundedness of $\|\tu\|_{L^2}$ -- a direct consequence of \eqref{X1a} and \eqref{T1}. This yields $\|\tu\|_{L^\infty}^6 \ls C\|\tu\tu_x\|_{L^2}^2$. Similarly, $\|\tu\|_{L^\infty}^4 \ls C\|\tu_x\|_{L^2}^2$. 
Using these estimates, we obtain
\begin{align}\label{ee48}
\Big( \int_a^b u^{4-\gamma} \tu^4\tv^2\,\mathrm{d}x\Big)^\frac12 \ls C
\left( \|\tu\tu_x\|_{L^2} + \|\tu_x\|_{L^2}\right).
\end{align}
Substituting \eqref{ee48} into \eqref{ee40} yields
\begin{align}\label{ee49}
|\mathsf{R}_{4,1}| \ls 11 \|\tu\tu_x\|_{L^2}^2 + C\int_a^b u^{\gamma-2}|\tu_x|^2\,\mathrm{d}x + C\Big(\int_a^b u^{\gamma-2}|\tu_x|^2\,\mathrm{d}x \Big)^\frac12 \|\tu_x\|_{L^2}.
\end{align}
By Young's inequality, $\mathsf{R}_{4,2}$ is estimated as 
\begin{align}\label{ee51}
|\mathsf{R}_{4,2}| \ls C\mathsf{Y}(t)\left(\|\tu\|_{L^4}^4 + 1\right).
\end{align}
Using \eqref{ee49} and \eqref{ee51}, we update \eqref{ee39} as
\begin{align}\label{ee52}
\frac{\mathrm{d}}{\mathrm{d}t} \|\tu\|_{L^4}^4 + \|\tu\tu_x\|_{L^2}^2 \ls C\mathsf{Y}(t) \|\tu\|_{L^4}^4 + C\left[\mathsf{Y}(t)+\mathsf{Z}(t)\right] + C\sqrt{\mathsf{Z}(t)}\, \|\tu_x\|_{L^2},
\end{align}
where $\mathsf{Z}(t)$ denotes the spatial integral of $u^{\gamma-2}|\tu_x|^2$, which is uniformly integrable in time according to \eqref{X5a}. Multiplying \eqref{ee52} by $\gamma\underline{\alpha}^{\gamma-4}$ and adding the result to \eqref{ee30} yields
\begin{align*}
&\, \frac{\mathrm{d}}{\mathrm{d}t} \left(2\mathcal{E}_2 + \|\tv\|_{L^2}^2 + \gamma \underline{\alpha}^{\gamma-4}\|\tu\|_{L^4}^4\right) + \frac{\gamma}{4} \underline{\alpha}^{\gamma-2}\|\tu_x\|_{L^2}^2\,\mathrm{d}x + \frac{\gamma}{2} \underline{\alpha}^{\gamma-4} \|\tu\tu_x\|_{L^2}^2 \notag \\
 \ls &\, C\mathsf{Y}(t) \|\tu\|_{L^4}^4 + C\left[\mathsf{Y}(t) +\mathsf{Z}(t)\right],
\end{align*}
where we also applied Young's inequality. Using Gr\"onwall's inequality, we obtain 
\begin{align}\label{ee56}
\|\tu(t)\|_{L^4}^2 + \int_0^t \big(\|\tu_x\|_{L^2}^2 + \|\tu\tu_x\|_{L^2}^2\big) \mathrm{d}\tau \ls C.
\end{align}
Similarly, we can establish this estimate when $3<\gamma\leqslant 4$ by using \eqref{ee35}.

\vspace{.05 in}

\noindent \underline{When $\gamma> 4$}, \eqref{ee42} is invalid. Instead, we take the $L^2$ inner product of \eqref{2.10} with $(\gamma+1)|\tu|^{\gamma-1}\tu$:
\begin{align}\label{ee57}
&\,\frac{\mathrm{d}}{\mathrm{d}t}\|\tu(t)\|_{L^{\gamma+1}}^{\gamma+1} + \gamma(\gamma+1)\big\||\tu|^{\frac{\gamma-1}{2}}\tu_x\big\|_{L^2}^2 + (\gamma+1)\int_a^b u |\tu|^{\gamma+1}\dx = \mathsf{R}_{5,1} + \mathsf{R}_{5,2},
\end{align}
where the two quantities on the right-hand side are defined as
\begin{align*}
\mathsf{R}_{5,1} \triangleq \gamma(\gamma+1)\int_a^b u |\tu|^{\gamma-1}\tu_x\tv \,\mathrm{d}x, \quad
\mathsf{R}_{5,2} \triangleq (\gamma+1)\int_a^b \left[\Psi\alpha_t - \alpha_t +u (1-\alpha)\right] |\tu|^{\gamma-1}\tu\,\mathrm{d}x.
\end{align*}
By H\"older's inequality and the boundedness of $\|\tv\|_{L^2}$, we can show that
\begin{align}\label{ee58}
|\mathsf{R}_{5,1}| \ls C[\mathsf{Z}(t)]^{\frac{1}{\gamma-2}} \|\tu_x\|_{L^2}^{\frac{\gamma-4}{\gamma-2}} \|\tu\|_{L^\infty}^{\gamma-1},
\end{align}
Using the uniform boundedness of $\|\tu\|_{L^\gamma}$ (see \eqref{X1b}), we infer that 
\begin{align}\label{ee59}
\|\tu\|_{L^\infty}^{\gamma-1} \ls (\gamma-1) \|\tu\|_{L^\gamma}^{\frac{\gamma}{2}} \Big(\int_a^b |\tu|^{\gamma-4} |\tu_x|^2 \,\mathrm{d}x\Big)^\frac12 \ls C \Big(\int_a^b |\tu|^{\gamma-4} |\tu_x|^2 \,\mathrm{d}x\Big)^\frac12.
\end{align}
Substituting \eqref{ee59} into \eqref{ee58} and applying Young's inequality, we obtain
\begin{align}\label{ee63}
|\mathsf{R}_{5,1}| \ls C\mathsf{Z}(t) + \delta \left(\big\||\tu|^{\frac{\gamma-1}{2}}\tu_x\big\|_{L^2}^2 + 2\|\tu_x\|_{L^2}^2 \right),
\end{align}
where $\delta>0$ is a constant. Substituting \eqref{ee63} into \eqref{ee57} and applying Young's inequality yield
\begin{align*}
\frac{\mathrm{d}}{\mathrm{d}t}\|\tu(t)\|_{L^{\gamma+1}}^{\gamma+1} + \left[\gamma(\gamma+1) -\delta\right] \big\||\tu|^{\frac{\gamma-1}{2}}\tu_x\big\|_{L^2}^2 \ls C\mathsf{Y}(t) \|\tu(t)\|_{L^{\gamma+1}}^{\gamma+1} + C\left[\mathsf{Y}(t) + \mathsf{Z}(t)\right] + 2\delta \|\tu_x\|_{L^2}^2.
\end{align*}
Combining this with \eqref{ee37} and choosing $\delta>0$ sufficiently small, we can show that
\begin{align}\label{ee65}
\|\tu(t)\|_{L^{\gamma+1}}^{\gamma+1} + \int_0^t  \left(\|\tu_x\|_{L^2}^2 + \big\||\tu|^{\frac{\gamma-1}{2}}\tu_x\big\|_{L^2}^2\right) \mathrm{d}\tau \ls C.
\end{align}
Next, we recover the dissipation for $\tv$.

\subsubsection{$H^1$-estimates}

The energy estimates in this subsection are derived in a manner similar to those in $\S$3.1, namely by obtaining a damped equation for $\tv_x$. However, the strong nonlinearity arising from the term $(u^\gamma)_x$ makes the technical details considerably more involved. For clarity, we outline only the main steps of the proof here, and relegate the detailed calculations to the Appendix. By analyzing the resulting damped equation, we can in fact show that
\begin{align}\label{C1}
\frac{\mathrm{d}}{\mathrm{d}t}\mathsf{H}(t) + \left(\underline{\alpha}^\gamma-\eta\right) \|\tv_x\|_{L^2}^2 \ls \eta \|\tu_{xx}\|_{L^2}^2 + C \mathsf{J}(t) \left(\|\tv_x\|_{L^2}^2 + \|\tu_x\|_{L^2}^2 + 1 \right),
\end{align}
where the quantities $\mathsf{H}(t)$ and $\mathsf{J}(t)$ are given by
$$
\begin{aligned}
\mathsf{H}(t) &\triangleq \frac{1}{2\gamma} \|\tv_x\|_{L^2}^2 - \int_a^b \alpha^{\gamma-1}\tu\tv_x\dx,\\[2mm]
\mathsf{J}(t) &\triangleq \left\{
\begin{aligned}
&\|\tu\tu_x\|_{L^2}^2 + \|\tu_x\|_{L^2}^2 + \mathsf{Y}(t),\quad 2\leqslant \gamma \leqslant 4,\\
&\big\||\tu|^{\frac{\gamma-1}{2}}\tu_x\big\|_{L^2}^2+ \|\tu_x\|_{L^2}^2 + \mathsf{Y}(t),\quad \gamma > 4.
\end{aligned}
\right.
\end{aligned}
$$
According to \eqref{BA-new}, \eqref{ee56} and \eqref{ee65}, the quantity $\mathsf{J}(t)$ is uniformly integrable in time. Note that the inequality \eqref{C1} is not closed for two obvious reasons: lack of control of $\|\tu_{xx}\|_{L^2}^2$ and potential non-positivity of $\mathsf{H}(t)$. We shall do a two-step repair to fix the issues. The first step is to combine \eqref{C1} with the $H^1$-estimate of $\tu$, which will wipe out $\eta\|\tu_{xx}\|_{L^2}^2$. The second step is to invoke the entropy-based estimate \eqref{X5} and utilize the super-quadratic structure of the expanded entropy functional to eliminate the possible negativity of $\mathsf{H}(t)$. 

{\bf Step 1.} Replacing $\beta(x,t)$ by $\Psi(t)$ in \eqref{A37}, we have
\begin{align*}
\frac{\mathrm{d}}{\mathrm{d}t}\|\tu_x\|_{L^2}^2 + \|\tu_{xx}\|_{L^2}^2 \leqslant \|(\tu\tv)_x + \Psi\tu_x + (\alpha\tv)_x + \Psi \alpha_x + u(1-\alpha) - u\tu - \alpha_t\|_{L^2}^2.
\end{align*}
from which one can show that  
\begin{align}\label{n58}
\frac{\mathrm{d}}{\mathrm{d}t} \|\tu_x\|_{L^2}^2 + \|\tu_{xx}\|_{L^2}^2 
\ls 2\overline{\alpha}^2\|\tv_x\|_{L^2}^2 + C\left(\mathsf{Y}(t) + \|\tu_x\|_{L^2}^2 \right)\left(\|\tv_x\|_{L^2}^2+1\right).
\end{align}
Multiplying \eqref{n58} by $\frac{\underline{\alpha}^\gamma}{4 \overline{\alpha}^2}$, adding the result to \eqref{C1}, and choosing $\eta=\min\left\{\frac{\underline{\alpha}^\gamma}{4}, \frac{\underline{\alpha}^\gamma}{8 \overline{\alpha}^2} \right\}$, we obtain
\begin{align}\label{n59}
\frac{\mathrm{d}}{\mathrm{d}t} \left(\mathsf{H}(t) + \frac{\underline{\alpha}^\gamma}{4\overline{\alpha}^2} \|\tu_x(t)\|_{L^2}^2\right) + \frac{\underline{\alpha}^\gamma}{4}\|\tv_x\|_{L^2}^2 + \frac{\underline{\alpha}^\gamma}{8 \overline{\alpha}^2} \|\tu_{xx}\|_{L^2}^2 \ls C \mathsf{J}(t) \left(\|\tv_x\|_{L^2}^2 + \|\tu_x\|_{L^2}^2 + 1 \right).
\end{align}

{\bf Step 2.} Dropping the dissipation term on the left-hand side of \eqref{X5}, we obtain  
\begin{align}\label{n64}
\frac{\mathrm{d}}{\mathrm{d}t} \left(2\mathcal{E}_2(t) + \|\tv(t)\|_{L^2}^2 \right) \ls C\mathsf{Y}(t).
\end{align}
According to \eqref{re}, the expanded entropy functional satisfies
\begin{align}\label{n65}
2\mathcal{E}_2(t) \gs \frac{\gamma \underline{\alpha}^{\gamma-2}}{\gamma-1} \|\tu(t)\|_{L^2}^2.
\end{align}
For $\chi=(\gamma-1)\overline{\alpha}^{2\gamma-2}\underline{\alpha}^{2-\gamma}$, using \eqref{n65}, it is straightforward to show that 
\begin{align}\label{n70}
2\chi\mathcal{E}_2(t) + \mathsf{H}(t) \geqslant \frac{1}{4\gamma} \|\tv_x(t)\|_{L^2}^2 + \frac{1}{4\gamma}\big\| (|\tv_x| - 2\gamma \overline{\alpha}^{\gamma-1}|\tu|)(t)\big\|_{L^2}^2.
\end{align}
Multiplying \eqref{n64} by $\chi$ and adding the result to \eqref{n59} yield
\begin{align*}
\frac{\mathrm{d}}{\mathrm{d}t} \mathsf{K}(t) + \frac{\underline{\alpha}^\gamma}{4}\|\tv_x\|_{L^2}^2 + \frac{\underline{\alpha}^\gamma}{8 \overline{\alpha}^2} \|\tu_{xx}\|_{L^2}^2 \ls C \mathsf{J}(t) \left(\mathsf{K}(t) + 1 \right).
\end{align*}
where the quantity $\mathsf{K}(t)$ is defined by 
$$
\mathsf{K}(t) \triangleq 2\chi \mathcal{E}_2(t) + \mathsf{H}(t) + \chi\|\tv(t)\|_{L^2}^2 +\frac{\underline{\alpha}^\gamma}{4\overline{\alpha}^2} \|\tu_x(t)\|_{L^2}^2.
$$
Gr\"onwall's inequality then yields
\begin{align}\label{n71}
\|\tu_x(t)\|_{L^2}^2 + \|\tv_x(t)\|_{L^2}^2 + \int_0^t \big(\|\tu_{xx}\|_{L^2}^2 +\|\tv_x\|_{L^2}^2\big)\mathrm{d}\tau \ls C,
\end{align}
where \eqref{n70} and the temporal integrability of $\mathsf{J}(t)$ have been applied. With \eqref{n71} at our disposal, the asymptotic analysis follows from $\S$2.1.4. This is the end of the proof of Theorem \ref{thm2} when $\varepsilon=0$.

\section{Conclusion}\label{C}

This study establishes the global stability of chemotaxis-logistic growth systems under dynamic Dirichlet boundary conditions, overcoming the challenge of unequal cell densities at spatial boundaries. For both parabolic-parabolic and parabolic-hyperbolic systems, we demonstrated:
\begin{itemize}
\item uniform boundedness of solutions in Sobolev norms for all time, ensuring biological realism;

\item time integrability of higher-order derivatives, reflecting damping of fluctuations; and

\item asymptotic convergence to the dynamic boundary conditions, with deviations decaying to zero.
\end{itemize}

\noindent Key technical innovations include:
\begin{itemize}
\item dynamic reference profiles that accommodates asymmetric and time-varying boundary inputs;

\item expanded entropy functionals yielding robust energy estimates under boundary forcing; and

\item domain splitting techniques to control quadratic nonlinearities from dynamic interpolation. 

\end{itemize}

\noindent Together, these provide a versatile framework for analyzing boundary-driven problems in applied fields where spatial asymmetry and temporal variation are critical, such as tissue interfaces, directed cell migration, or engineered microenvironments.

\vspace{.05 in}

Several promising research directions emerge:
\begin{itemize}
\item Generalizing boundary conditions to Neumann/Robin or mixed types.

\item Studying pattern formation in 2D/3D geometries with dynamic boundaries.

\item Incorporating stochastic fluctuations in boundary data or internal dynamics.

\item Extending analysis to multi-species chemotaxis under asymmetric boundary forcing.

\item Designing optimal boundary control protocols for therapeutic or bioengineering applications.
\end{itemize}

\noindent Bridging mathematical rigor with biological relevance, this work opens new avenues for modeling dynamically controlled chemotactic processes, paving the way for advances in developmental biology, immunology, and tissue engineering.

\section*{Acknowledgements} 

Support of this work came partially from the Fundamental Research Funds for Central Universities of China No.\,3072024WD2401 (K. Zhao); National Science Foundation through DMS 2316699 (P. Fuster Aguilera); National Science Foundation through DMS 2213363, DMS 2206491, DMS 2511403, the Simons Foundation through MP-TSM 00014320, as well as the Dolciani Halloran Foundation (V.R. Martinez)

\section*{Appendix}

This appendix contains the technical details leading to \eqref{C1}. Similar to \eqref{2.12}, we have 
\begin{align}\label{n2b}
\tv_t - (u^\gamma)_x = -\Psi'(t).
\end{align}
Differentiating \eqref{n2b} with respect to $x$ and dividing by $\gamma$, we obtain
\begin{align}\label{n7}
\gamma^{-1} \tv_{xt} - (u^{\gamma-1} - \alpha^{\gamma-1}) \tu_{xx} - (\gamma-1) u^{\gamma-2} |u_x|^2 = \alpha^{\gamma-1} \tu_{xx}.
\end{align}
Multiplying \eqref{2.10} by $\alpha^{\gamma-1}$ gives
\begin{align}\label{n8}
&\,\alpha^{\gamma-1}\tu_t - \alpha^{\gamma-1}(\tu\tv)_x - \alpha^{\gamma}\tv_x - \alpha^{\gamma-1}\alpha_x\tv - \alpha^{\gamma-1}\Psi\tu_x - \alpha^{\gamma-1}\Psi \alpha_x \notag\\
=&\, \alpha^{\gamma-1}\tu_{xx} + \alpha^{\gamma-1} u(1-\alpha) - \alpha^{\gamma-1} u \tu - \alpha^{\gamma-1}\alpha_t.
\end{align}
Combining\eqref{n7} and \eqref{n8}, we arrive at
\begin{align}\label{n9}
\gamma^{-1}\tv_{xt} + \alpha^{\gamma}\tv_x = &\, (u^{\gamma-1} - \alpha^{\gamma-1}) \tu_{xx} + (\gamma-1) u^{\gamma-2} |u_x|^2 + \alpha^{\gamma-1}\tu_t - \alpha^{\gamma-1}(\tu\tv)_x - \alpha^{\gamma-1}\alpha_x\tv \notag\\
&\, - \alpha^{\gamma-1}\Psi \tu_x - \alpha^{\gamma-1}\Psi \alpha_x - \alpha^{\gamma-1} u(1-\alpha) + \alpha^{\gamma-1} u \tu + \alpha^{\gamma-1}\alpha_t .
\end{align}
Taking the $L^2$ inner product of \eqref{n9} with $\tv_x$ and rearranging terms, we can show that
\begin{align}\label{n13}
&\,\frac{\mathrm{d}}{\mathrm{d}t} \left( \frac{1}{2\gamma}\|\tv_x\|_{L^2}^2 - \int_a^b \alpha^{\gamma-1}\tu\tv_x\dx \right) + \int_a^b \alpha^\gamma |\tv_x|^2\dx = \sum_{k=1}^{11} \mathsf{R}_{6,k},
\end{align}
where the quantities on the right-hand side are given by
\begin{align*}
\mathsf{R}_{6,1} &= \int_a^b (u^{\gamma-1} - \alpha^{\gamma-1}) \tu_{xx} \tv_x\dx, \quad \mathsf{R}_{6,2} = (\gamma-1) \int_a^b u^{\gamma-2}|u_x|^2 \tv_x\dx, \\  
\mathsf{R}_{6,3} &= -(\gamma-1)\int_a^b \alpha^{\gamma-2}\alpha_t \tu \tv_x\dx, \quad \mathsf{R}_{6,4} = (\gamma-1)\alpha_x\int_a^b \alpha^{\gamma-2} \tu (u^\gamma)_x\dx, \\
\mathsf{R}_{6,5} &= -(\gamma-1)\alpha_x \Psi' \int_a^b \alpha^{\gamma-2} \tu \dx, \quad \mathsf{R}_{6,6} = \int_a^b \alpha^{\gamma-1} \tu_x (u^\gamma)_x \dx, \quad \mathsf{R}_{6,7} = -\Psi' \int_a^b \alpha^{\gamma-1} \tu_x \dx, \\ 
\mathsf{R}_{6,8} &= - \int_a^b \alpha^{\gamma-1}(\tu\tv)_x \tv_x\dx, \quad \mathsf{R}_{6,9} = - \alpha_x \int_a^b \alpha^{\gamma-1} \tv \tv_x \dx, \quad  
\mathsf{R}_{6,10} = \Psi\int_a^b \alpha^{\gamma-1}\tu_x\tv_x\dx, \\
\mathsf{R}_{6,11} &= - \Psi \alpha_x\int_a^b \alpha^{\gamma-1}\tv_x\dx, \quad \mathsf{R}_{6,12} = \int_a^b \alpha^{\gamma-1}u(u-1)\tv_x \dx, \quad \mathsf{R}_{6,13} = \int_a^b \alpha^{\gamma-1}\alpha_t\tv_x\dx.
\end{align*}

\subsection{Estimates of $R_k$'s when $2\leqslant \gamma \leqslant 3$}
By the mean value theorem, $u^{\gamma-1} - \alpha^{\gamma-1} = (\gamma-1) w^{\gamma-2} \tu $, where $w$ is between $u$ and $\alpha$. Since $u\gs 0$ and $\alpha \gs \underline{\alpha}>0$, we must have $0\ls w \ls |\tu| + \alpha$. This, along with $2\leqslant \gamma \leqslant 3$, implies $w^{\gamma-2}\ls (|\tu| + \alpha)^{\gamma-2}\ls |\tu|^{\gamma-2} + \alpha^{\gamma-2}$. Hence, 
\begin{align}\label{n15}
\left|u^{\gamma-1} - \alpha^{\gamma-1}\right| \ls (\gamma-1) \left(|\tu|^{\gamma-2} + \alpha^{\gamma-2}\right)|\tu|.
\end{align}
Since $\|\tu\|_{L^2}$ is uniformly bounded, it holds that $\|\tu\|_{L^\infty}\lesssim \|\tu_x\|_{L^2}^\frac12$. By \eqref{n15}, we have
\begin{align*}
|\mathsf{R}_{6,1}| \ls C \Big(\|\tu_x\|_{L^2}^{\frac{\gamma-1}{2}} + \|\tu_x\|_{L^2}^\frac12\Big) \|\tu_{xx}\|_{L^2}\|\tv_x\|_{L^2}.
\end{align*}
Since $\gamma-1\in [1,2]$, by Young's inequality, 
\begin{align}\label{n17}
|\mathsf{R}_{6,1}| \ls \frac{\eta}{4} \|\tu_{xx}\|_{L^2}^2 + \frac{\eta}{4} \|\tv_x\|_{L^2}^2 + C\|\tu_x\|_{L^2}^2 \|\tv_x\|_{L^2}^2,
\end{align}
where $\eta>0$ is a constant to be determined. Using $u^{\gamma-2} \ls |\tu|^{\gamma-2} + \alpha^{\gamma-2}$, we can show that 
\begin{align}\label{n18}
|\mathsf{R}_{6,2}| \ls C\|u_x\|_{L^\infty} \left( \big\| |\tu|^{\gamma-2}u_x\big\|_{L^2} +\|u_x\|_{L^2}\right) \|\tv_x\|_{L^2}.
\end{align}
Since $\tu_x$ is mean free, it holds that $\|\tu_x\|_{L^\infty}^2 \lesssim \|\tu_x\|_{L^2} \|\tu_{xx}\|_{L^2}$. Moreover, since $2(\gamma-2)\in[0,2]$, we must have $|\tu|^{2(\gamma-2)} \leqslant |\tu|^2+1$. With such estimates, we update \eqref{n18} as 
\begin{align}\label{n19}
|\mathsf{R}_{6,2}| \ls C\Big(\|\tu_x\|_{L^2}^\frac12 \|\tu_{xx}\|_{L^2}^\frac12 + |\alpha_2-\alpha_1| \Big) \left(\|\tu\tu_x\|_{L^2} + \|\tu_x\|_{L^2} + |\alpha_2-\alpha_1| \right)\|\tv_x\|_{L^2}.
\end{align}
For the quantity on the right-hand side of \eqref{n19}, by Young's inequality, we have 
\begin{align*}
&\,C\|\tu_x\|_{L^2}^\frac12 \|\tu_{xx}\|_{L^2}^\frac12 \left(\|\tu\tu_x\|_{L^2} + \|\tu_x\|_{L^2} + |\alpha_2-\alpha_1| \right)\|\tv_x\|_{L^2} \notag\\
\ls &\, \frac{\eta}{4} \|\tu_{xx}\|_{L^2}^2 + C \left(\|\tu\tu_x\|_{L^2}^2 + \|\tu_x\|_{L^2}^2 + |\alpha_2-\alpha_1| \right) \left( \|\tv_x\|_{L^2}^2 + 1\right).
\end{align*}
This gives us an upgraded version of \eqref{n19}: 
\begin{align*}
|\mathsf{R}_{6,2}| \ls \frac{\eta}{4} \|\tu_{xx}\|_{L^2}^2 + C \left(\|\tu\tu_x\|_{L^2}^2 + \|\tu_x\|_{L^2}^2 + |\alpha_2-\alpha_1| \right) \left(\|\tv_x\|_{L^2}^2 + 1 \right).
\end{align*}
Moreover, using the uniform boundedness of $\|\tu\|_{L^2}$, we can show that
\begin{align*}
|\mathsf{R}_{6,3}| \ls C\left(|\alpha_1'| + |\alpha_2'|\right) \left(\|\tv_x\|_{L^2}^2+1\right). 
\end{align*}
Since $0\ls u^{\gamma-1} \ls 2^{\gamma-2}(|\tu|^{\gamma-1} + \alpha^{\gamma-1})$, we infer that 
\begin{align}\label{n23}
|\mathsf{R}_{6,4}| &\ls C|\alpha_2-\alpha_1| \left(\|\tu\|_{L^\gamma}^\gamma \left(\|\tu_x\|_{L^\infty} + |\alpha_x|\right) + \|\tu\|_{L^2}\|\tu_x\|_{L^2} + \|\tu\|_{L^2} |\alpha_x| \right)\notag\\
&\ls C|\alpha_2-\alpha_1| \big(\|\tu_{xx}\|_{L^2} + \|\tu_x\|_{L^2} + 1 \big),
\end{align}
where the boundedness of $\|\tu\|_{L^\gamma}$ and Poincar\'e's inequality are applied. This implies
\begin{align*}
|\mathsf{R}_{6,4}| \ls \frac{\eta}{4} \|\tu_{xx}\|_{L^2}^2 + C \left(\|\tu_x\|_{L^2}^2 + |\alpha_2-\alpha_1| \right).
\end{align*}
Using the uniform boundedness of $\Psi'(t)$ and $\|\tu\|_{L^2}$, we deduce that 
\begin{align*}
|\mathsf{R}_{6,5}| \ls C|\alpha_2-\alpha_1| \|\tu\|_{L^2} \ls C|\alpha_2-\alpha_1|.
\end{align*}
The estimate of $\mathsf{R}_{6,6}$ is similar to \eqref{n23}:
\begin{align*}
|\mathsf{R}_{6,6}| &\ls C \Big(\|\tu\|_{L^\gamma}^{\gamma-1}\|\tu_x\|_{L^\infty}^2 + \|\tu\|_{L^\gamma}^{\gamma-1} |\alpha_x| + \|\tu_x\|_{L^2}^2 + \|\tu_x\|_{L^2}|\alpha_x|\Big) \notag \\
&\ls \frac{\eta}{4}\|\tu_{xx}\|_{L^2}^2 + C\left(\|\tu_x\|_{L^2}^2 + |\alpha_2-\alpha_1|\right).
\end{align*}
By the mean value theorem and boundedness of $\alpha_i$, we can show that
\begin{align*}
|\mathsf{R}_{6,7}| \ls C\left|\alpha_1^\gamma - \alpha_2^\gamma\right| \|\tu_x\|_{L^2} \ls C\left(\|\tu_x\|_{L^2}^2 + |\alpha_1-\alpha_2|\right).
\end{align*}
Using Sobolev embedding and Poincar\'e's inequality, we infer that 
\begin{align*}
|\mathsf{R}_{6,8}| \ls C\|\tu_x\|_{L^2}\|\tv_x\|_{L^2}^2 \ls \frac{\eta}{4}\|\tv_x\|_{L^2}^2 + C\|\tu_x\|_{L^2}^2\|\tv_x\|_{L^2}^2.
\end{align*}
The estimate of $\mathsf{R}_{6,9}$ follows from Poincar\'e's inequality:
\begin{align*}
|\mathsf{R}_{6,9}| \ls C|\alpha_2-\alpha_1| \|\tv\|_{L^2}\|\tv_x\|_{L^2} \ls C|\alpha_2-\alpha_1| \|\tv_x\|_{L^2}^2.
\end{align*}
Since $\Psi$ is uniformly bounded, we have
\begin{align*}
|\mathsf{R}_{6,10}| \ls C\|\tu_x\|_{L^2} \|\tv_x\|_{L^2} \ls \frac{\eta}{4}\|\tv_x\|_{L^2}^2 + C\|\tu_x\|_{L^2}^2.
\end{align*}
Similarly, it holds that
\begin{align*}
|\mathsf{R}_{6,11}| \ls C|\alpha_2-\alpha_1| \|\tv_x\|_{L^2} \ls C|\alpha_2-\alpha_1| \left(\|\tv_x\|_{L^2}^2+1\right).
\end{align*}
Using the boundedness of $\|\tu\|_{L^2}$ and Poincar\'e's inequality, we deduce that 
\begin{align*}
|\mathsf{R}_{6,12}| \ls \frac{\eta}{4}\|\tv_x\|_{L^2}^2 + C\left( \|\tu_x\|_{L^2}^2 + |1-\alpha_1| + |\alpha_2-\alpha_1| \right).
\end{align*}
Similar as the estimate of $\mathsf{R}_{6,11}$, we have
\begin{align*}
|\mathsf{R}_{6,13}| \ls C\left(|\alpha_1'| + |\alpha_2'|\right) \left(\|\tv_x\|_{L^2}^2+1\right).
\end{align*}
Substituting the estimates of the $\mathsf{R}_{6,k}$'s into \eqref{n13} yields 
\eqref{C1}.

\subsection{Estimates of $R_k$'s when $3< \gamma\ls 4$}

We begin by noticing that \eqref{n17} is not valid when $3< \gamma\ls 4$. To resolve this issue, we re-derive an estimate for $\|\tu\|_{L^\infty}^{\gamma-1}$ as:
\begin{align*}
\|\tu\|_{L^\infty}^{\gamma-1} \ls (\gamma-1)\int_a^b |\tu|^{\gamma-2}|\tu_x|\dx \ls C\big\||\tu|^{\gamma-3}\tu_x\big\|_{L^2} \ls C\left(\|\tu\tu_x\|_{L^2}+\|\tu_x\|_{L^2}\right).
\end{align*}
Using this and \eqref{n15}, we re-estimate $\mathsf{R}_{6,1}$ as 
\begin{align}\label{n36a}
|\mathsf{R}_{6,1}| \ls \frac{\eta}{4} \|\tu_{xx}\|_{L^2}^2 + C\left( \|\tu\tu_x\|_{L^2}^2 +\|\tu_x\|_{L^2}^2 \right)\|\tv_x\|_{L^2}^2.
\end{align}
Next, using the inequality $u^{\gamma-2} \leqslant |\tu|^{\gamma-2} + \alpha^{\gamma-2}$, we infer that
\begin{align}\label{n36}
|\mathsf{R}_{6,2}| \ls C\Big(\int_a^b |\tu|^{\gamma-2} (\tu_x + \alpha_x)^2 |\tv_x|\dx + \int_a^b (\tu_x + \alpha_x)^2 |\tv_x|\dx \Big).
\end{align}
The estimate of the first integral on the right-hand side of \eqref{n36} is different from the previous case, due to the inequality $|\tu|^{2(\gamma-2)} < |\tu|^2+1$ is not valid when $3< \gamma \ls 4$. Instead, we proceed as follows:
\begin{align}\label{n37}
C\int_a^b |\tu|^{\gamma-2} (\tu_x + \alpha_x)^2 |\tv_x|\dx \ls C\|\tu\|_{L^\infty}\|\tu_x+\alpha_x\|_{L^\infty}\int_a^b |\tu|^{\gamma-3} \big(|\tu_x| + |\alpha_x|\big) |\tv_x|\dx.
\end{align}
The integral on the right-hand side is controlled, by using $|\tu|^{2(\gamma-3)} < |\tu|^2+1$, as 
\begin{align}\label{n38}
\int_a^b |\tu|^{\gamma-3} \left(|\tu_x| + |\alpha_x|\right) |\tv_x|\dx \ls C\left(\|\tu\tu_x\|_{L^2} + \|\tu_x\|_{L^2} + |\alpha_2-\alpha_1|\right)\|\tv_x\|_{L^2},
\end{align}
where the uniform boundedness of $\|\tu\|_{L^2}$ and $\alpha_i$ is applied. Moreover, since $\|\tu\|_{L^\infty}^2 \lesssim \|\tu\|_{L^2}\|\tu_x\|_{L^2}\lesssim \|\tu_x\|_{L^2}$ and $\|\tu_x\|_{L^\infty}^2 \lesssim \|\tu_x\|_{L^2} \|\tu_{xx}\|_{L^2}$, by \eqref{n38}, we update \eqref{n37} as 
\begin{align}\label{n40}
&\,C\int_a^b |\tu|^{\gamma-2} (\tu_x + \alpha_x)^2 |\tv_x|\dx \notag\\
\ls &\,C\left(\|\tu_x\|_{L^2} \|\tu_{xx}\|_{L^2}^\frac12 + \|\tu_x\|_{L^2}^\frac12 |\alpha_x|\right)\left(\|\tu\tu_x\|_{L^2} + \|\tu_x\|_{L^2} + |\alpha_2-\alpha_1|\right)\|\tv_x\|_{L^2}.
\end{align}
Using Young's inequality, we can show that 
\begin{align}\label{n41}
&\,C \|\tu_x\|_{L^2} \|\tu_{xx}\|_{L^2}^\frac12 \left(\|\tu\tu_x\|_{L^2} + \|\tu_x\|_{L^2} + |\alpha_2-\alpha_1|\right)\|\tv_x\|_{L^2}\notag\\
\ls &\,\frac{\eta}{8}\|\tu_{xx}\|_{L^2}^2 + C\left(\|\tu\tu_x\|_{L^2}^2 + \|\tu_x\|_{L^2}^2 + |\alpha_2-\alpha_1|\right) \left(\|\tv_x\|_{L^2}^2+\|\tu_x\|_{L^2}^2\right).
\end{align}
Moreover, it holds that 
\begin{align}\label{n42}
&\,C\|\tu_x\|_{L^2}^\frac12|\alpha_x| \left(\|\tu\tu_x\|_{L^2} + \|\tu_x\|_{L^2} + |\alpha_2-\alpha_1|\right)\|\tv_x\|_{L^2}\notag\\
\ls &\, C\left(\|\tu\tu_x\|_{L^2}^2 + \|\tu_x\|_{L^2}^2 + |\alpha_2-\alpha_1|\right) \left(\|\tv_x\|_{L^2}^2+1\right).
\end{align}
Substituting \eqref{n41}--\eqref{n42} into \eqref{n40} gives us 
\begin{align}\label{n43}
&\,C\int_a^b |\tu|^{\gamma-2} (\tu_x + \alpha_x)^2 |\tv_x|\dx \notag\\
\ls &\,\frac{\eta}{8}\|\tu_{xx}\|_{L^2}^2 + C\left(\|\tu\tu_x\|_{L^2}^2 + \|\tu_x\|_{L^2}^2 + |\alpha_2-\alpha_1|\right) \left(\|\tv_x\|_{L^2}^2+\|\tu_x\|_{L^2}^2+1\right).
\end{align}
The estimate of the second integral on the right-hand side of \eqref{n36} is similar as before:
\begin{align}\label{n44}
C \int_a^b (\tu_x + \alpha_x)^2 |\tv_x|\dx \ls \frac{\eta}{8}\|\tu_{xx}\|_{L^2}^2 + C \left(\|\tu_x\|_{L^2}^2 + |\alpha_2-\alpha_1|\right) \left(\|\tv_x\|_{L^2}^2+1\right).
\end{align}
Substituting \eqref{n43}--\eqref{n44} into \eqref{n36}, we arrive at 
\begin{align*}
|\mathsf{R}_{6,2}| \ls \frac{\eta}{4} \|\tu_{xx}\|_{L^2}^2 + C\left(\|\tu\tu_x\|_{L^2}^2 + \|\tu_x\|_{L^2}^2 + |\alpha_2-\alpha_1|\right) \left(\|\tv_x\|_{L^2}^2+\|\tu_x\|_{L^2}^2+1\right).
\end{align*}
The estimates of the remaining $R_k$'s are identical to those in $\S$4.2.

\subsection{Estimates of $R_k$'s when $\gamma >4$}

We first re-derive an estimate for $\|\tu\|_{L^\infty}^{\gamma-1}$:
\begin{align*}
\|\tu\|_{L^\infty}^{\gamma-1} \ls (\gamma-1)\int_a^b |\tu|^{\gamma-2}|\tu_x|\dx \ls C\big\||\tu|^{\frac{\gamma}{2}-2}\tu_x\big\|_{L^2} \ls C\left(\big\||\tu|^{\frac{\gamma-1}{2}}\tu_x\big\|_{L^2} + \|\tu_x\|_{L^2}\right),
\end{align*}
where the uniform boundedness of $\|\tu\|_{L^\gamma}$ and Young's inequality are applied. Similar to \eqref{n36a}, we have
\begin{align*}
|\mathsf{R}_{6,1}| \ls \frac{\eta}{4} \|\tu_{xx}\|_{L^2}^2 + C\left(\big\||\tu|^{\frac{\gamma-1}{2}}\tu_x\|_{L^2} +\|\tu_x\big\|_{L^2}^2 \right)\|\tv_x\|_{L^2}^2.
\end{align*}
For $\mathsf{R}_{6,2}$, first of all, since $|\tu|^{\gamma-2} < |\tu|^{\gamma-1} +1$, we revise \eqref{n36} as 
\begin{align}\label{n49}
|\mathsf{R}_{6,2}| \ls C\Big(\int_a^b |\tu|^{\gamma-1} (\tu_x + \alpha_x)^2 |\tv_x|\dx + \int_a^b (\tu_x + \alpha_x)^2 |\tv_x|\dx \Big).
\end{align}
By H\"older's inequality, the first integral on the right-hand side of \eqref{n49} satisfies
\begin{align}\label{n50}
C\int_a^b |\tu|^{\gamma-1} (\tu_x + \alpha_x)^2 |\tv_x|\dx \ls C \|\tu\|_{L^\infty}^{\frac{\gamma-1}{2}} \|\tu_x+\alpha_x\|_{L^\infty} \big\||\tu|^{\frac{\gamma-1}{2}}(\tu_x+\alpha_x)\big\|_{L^2} \|\tv_x\|_{L^2}.
\end{align}
Using the uniform boundedness of $\|\tu\|_{L^\gamma}$, we estimate the product on the right-hand side of \eqref{n50} as 
\begin{align*}
&\,C\|\tu\|_{L^\infty}^{\frac{\gamma-1}{2}} \|\tu_x+\alpha_x\|_{L^\infty} \big\||\tu|^{\frac{\gamma-1}{2}}(\tu_x+\alpha_x)\big\|_{L^2} \|\tv_x\|_{L^2}\notag\\
\ls &\, C\Big(\big\||\tu|^{\frac{\gamma-1}{2}}\tu_x\big\|_{L^2} + \|\tu_x\|_{L^2}\Big)^\frac12 \Big(\|\tu_x\|_{L^2}^\frac12\|\tu_{xx}\|_{L^2}^\frac12 + |\alpha_x|\Big) \Big(\big\||\tu|^{\frac{\gamma-1}{2}}\tu_x\big\|_{L^2} + |\alpha_x|\Big) \|\tv_x\|_{L^2}.
\end{align*}
It follows from Young's inequality that
\begin{align*}
&\,C \Big(\big\||\tu|^{\frac{\gamma-1}{2}}\tu_x\big\|_{L^2} + \|\tu_x\|_{L^2}\Big)^\frac12 \|\tu_x\|_{L^2}^\frac12\|\tu_{xx}\|_{L^2}^\frac12 \Big(\big\||\tu|^{\frac{\gamma-1}{2}}\tu_x\big\|_{L^2} + |\alpha_x|\Big) \|\tv_x\|_{L^2}\notag\\
\ls &\, \frac{\eta}{8} \|\tu_{xx}\|_{L^2}^2 + C \Big(\big\||\tu|^{\frac{\gamma-1}{2}}\tu_x\big\|_{L^2}^2 + \|\tu_x\|_{L^2}^2 + |\alpha_2-\alpha_1|\Big) \Big(\|\tv_x\|_{L^2}^2 + \|\tu_x\|_{L^2}^2\Big).
\end{align*}
Moreover, we have 
\begin{align}\label{n53}
&\,C |\alpha_x| \Big(\big\||\tu|^{\frac{\gamma-1}{2}}\tu_x\big\|_{L^2} + \|\tu_x\|_{L^2}\Big)^\frac12 \Big(\big\||\tu|^{\frac{\gamma-1}{2}}\tu_x\big\|_{L^2} + |\alpha_x|\Big) \|\tv_x\|_{L^2}\notag\\
\ls &\, C\Big(\big\||\tu|^{\frac{\gamma-1}{2}}\tu_x\big\|_{L^2}^2 + \|\tu_x\|_{L^2}^2 + |\alpha_2-\alpha_1|\Big) \left(\|\tv_x\|_{L^2}^2+1\right).
\end{align}
The combination of \eqref{n50} and \eqref{n53} gives
\begin{align*}
&\,C\int_a^b |\tu|^{\gamma-1} (\tu_x + \alpha_x)^2 |\tv_x|\dx \notag\\
\ls &\,\frac{\eta}{8} \|\tu_{xx}\|_{L^2}^2 + C \Big(\big\||\tu|^{\frac{\gamma-1}{2}}\tu_x\|_{L^2}^2 + \|\tu_x\big\|_{L^2}^2 + |\alpha_2-\alpha_1|\Big) \Big(\|\tv_x\|_{L^2}^2 + \|\tu_x\|_{L^2}^2+1\Big),
\end{align*}
which, together with \eqref{n44}, yields
\begin{align*}
|\mathsf{R}_{6,2}| \ls \frac{\eta}{4} \|\tu_{xx}\|_{L^2}^2 + C \Big(\big\||\tu|^{\frac{\gamma-1}{2}}\tu_x\big\|_{L^2}^2 + \|\tu_x\|_{L^2}^2 + |\alpha_2-\alpha_1|\Big) \left(\|\tv_x\|_{L^2}^2 + \|\tu_x\|_{L^2}^2+1\right).
\end{align*}
Again, the estimates of the remaining $R_k$'s are identical to those in $\S$4.2. 

\

\end{document}